\documentclass[reqno,11pt]{amsart}
\usepackage{graphicx,  enumitem, amsmath, amssymb,amsthm,hyperref,amscd,xcolor,bbm}
\usepackage{mathabx,manfnt,braket,mathrsfs, comment,mathtools}
\usepackage{tabularx}
\usepackage{bm}

\usepackage{tikz}
\usepackage{tikz-cd}
\usepackage{cancel} 
\usetikzlibrary{arrows, matrix,decorations.pathmorphing}
\usepackage[utf8]{inputenc}
   

\numberwithin{equation}{section}

\newtheorem*{thm*}{Theorem}
\newtheorem{prop}[equation]{Proposition}

\newtheorem{lem}[equation]{Lemma}
\theoremstyle{definition}

\newcommand{\Uu}{\mathscr{U}}

\newcommand{\abs}[1]{\left\vert#1\right\vert}
\newcommand{\cx}{{\mathbb{C}}}
\newcommand{\rl}{{\mathbb{R}}}

\newcommand{\D}{\mathbb{D}}
\newcommand{\Z}{\mathbb{Z}}
\newcommand{\N}{\mathbb{N}}

\newcommand{\T}{\mathbb{T}}
\newcommand{\Q}{\mathbb{Q}}

 \DeclareMathOperator{\lcm}{lcm}

\newcommand{\one}{\mathbbm{1}}

\newcommand{\dd}{\mathrm{D}}

\newcommand{\ol}{\overline}

\newcommand{\Identity}{\normalfont\large I}
\newcommand{\rvline}{\hspace*{-\arraycolsep}\vline\hspace*{-\arraycolsep}}

\DeclareMathOperator{\adj}{adj}
\DeclareMathOperator{\diag}{diag}

\title{Bergman kernels of Monomial Polyhedra}
\author[D. Chakrabarti]{Debraj Chakrabarti}
\urladdr{https://people.se.cmich.edu/chakr2d/}

\author[I. Cinzori]{Isaac Cinzori}

\author[I. Gaidhane]{Ishani Gaidhane}

\author[J. Gregory]{Jonathan Gregory}

\author[M. Wright]{Mary Wright}
\email{chakr2d@cmich.edu, cinzo1im@cmich.edu, gaidh1iv@cmich.edu, grego3j@cmich.edu,  wrigh5me@cmich.edu }

\address{Department of Mathematics, Central Michigan University, Mt Pleasant, MI 48859, USA.}
\thanks{Chakrabarti, Gaidhane, Gregory and Wright were supported in part by US National Science
Foundation grant number DMS-2153907. Chakrabarti was also partially supported by a gift from the Simons 
Foundation (number 706445).}
\subjclass[2020]{32A25}

\begin{document}
\maketitle
\begin{abstract} 
The Bergman kernels of monomial polyhedra are explicitly computed.
Monomial polyhedra are a class of bounded pseudoconvex Reinhardt domains defined as sublevel sets of  Laurent monomials. Their kernels are rational functions and are obtained by an application of Bell's transformation formula. 
\end{abstract}

\section{Introduction} 
We will say that a \emph{bounded} domain (open, connected subset) $\Uu_B\subset \cx^n, n\geq 2$ is a \emph{monomial polyhedron} if there is 
an $n\times n$ matrix of integers $B$ such that 
\begin{equation}
    \label{eq-ubdef}
    \Uu_B=\left\{(z_1,\dots, z_n)\in \cx^n: \text{ for each } 1\leq j \leq n, \quad
    \abs{z_1}^{b^j_1}\dots \abs{z_n}^{b^j_n}<1\right\},
\end{equation}
where $b^j_k\in \Z$ denotes the entry at the $j$-th row and $k$-th column of $B$, where it is assumed in \eqref{eq-ubdef} that the power $\abs{z_k}^{b^j_k}$ is well-defined for each $z\in \Uu_B$ and each $1\leq j, k \leq n$,
i.e., if $b^j_k<0$ for some $j,k$, then $z_k\not=0$ for each point $z\in \Uu_B$.
We summarize the situation by saying that $B$ is the defining matrix of the domain $\Uu_B$, or simply that $\Uu_B$ is defined by $B$. Monomial polyhedra are clearly Reinhardt, and looking at their log absolute image, we see immediately that they are also pseudoconvex. 
If $B=I$ is the $n\times n$ identity matrix, then $\Uu_I$ is the unit polydisc, which may be regarded as a trivial example.
A  famous nontrivial example of a monomial polyhedron is the \emph{classical Hartogs triangle}:
\[ \{(z_1,z_2)\in \cx^2: \abs{z_1}<\abs{z_2}<1\},\] a venerable source 
of counterexamples in complex analysis, which is easily seen to be a monomial polyhedron defined by the matrix $\begin{pmatrix} 1&-1\\0&\phantom{-}1\end{pmatrix}.$
Pathologies of the Hartogs triangle (e.g., lack of Stein neighborhood bases) generalize to nontrivial monomial polyhedra, explaining the importance of these domains in complex analysis. Monomial polyhedra and domains closely associated to them have been studied extensively in complex and harmonic analysis (see \cite{nagelduke, nagelpramanik, summer20}).

In the last decade, there has been activity surrounding domains generalizing the Hartogs triangle, their Bergman kernels, and the Bergman projection on these domains. (For general information on the Bergman kernel and projection, see, e.g. \cite{krantzbergman, zhubergman}). The current interest began with the discovery in \cite{chakzeytuncu,EdMc1,EdhMcN16} of remarkable $L^p$-mapping properties of the Bergman projection on the \emph{Generalized Hartogs Triangle}, the domain $\Omega_\gamma=\{\abs{z_1}^\gamma<\abs{z_2}<1\}, \gamma >0$, which is a monomial polyhedron if $\gamma$ happens to be rational. 
Since a first step in investigating the Bergman projection is to understand the Bergman kernel, a series of studies have been directed toward the goal of obtaining the Bergman kernel of various generalizations of the Hartogs triangle. The kernel of the classical Hartogs triangle has been known for a long time and occurs explicitly in \cite{bremmerman}.
In \cite{lukethesis, Edh16}, the kernel of the Generalized Hartogs Triangle $\Omega_\gamma$ was obtained when $\gamma$ is either a positive integer or the reciprocal of a positive integer. In \cite{park}, the kernel of the domain $\{\abs{z_1}^{k_1}<\abs{z_2}^{k_2}< \abs{z_3}^{k_3}\}\subset \cx^3$ was obtained. In \cite{pjm,Chen17,zhang1} other types of generalizations of the Hartogs triangle were studied and explicit kernels were obtained. In each of these works, ad hoc methods based on Bell's transformation formula were used and led to rather complicated expressions for the kernel. 

The original motivating problem of determining the values of $p$ for which the Bergman projection is bounded in $L^p$-norm can be studied once the kernel is known (see \cite{Chen17,zhang1,zhang2,ChEdMc19} etc., see also the survey \cite{zeytuncu20}). 
In \cite{summer20}, the problem was studied in the context of the monomial polyhedra \eqref{eq-ubdef}, using a representation of these domains as quotients, thus avoiding the question of computing the Bergman kernel (similar ideas are found in \cite{chakzeytuncu, CKY19}). While the Bergman projection itself is not bounded in $L^p$ on a monomial polyhedron if $p$ does not lie in a certain bounded interval, it was shown in \cite{mbp} how to construct an alternative bounded projection operator from $L^p(\Uu_B)$ to its holomorphic subspace.

 In \cite{summer20}, it was shown using Galois theory that the Bergman kernels of
the domains $\Uu_B$ are rational functions of the coordinates (this is also related to the results of  \cite{bellrat}). The question naturally arises of explicitly computing the Bergman kernel of 
$\Uu_B$ in terms of the $n\times n$ integer matrix $B$. When $n=2$, this question was answered in the recent preprint  \cite{almughrabi}. However, some of the ideas of \cite{almughrabi} are specific to the two-dimensional case and do not generalize easily to higher dimensions. In this paper, we explicitly compute the Bergman kernel of the domain \eqref{eq-ubdef} in terms of the matrix $B$ for any $n\geq 2$, obtaining formula \eqref{eq-kernel2} below.
This formula is of interest from many points of view. First, it adds infinite examples to the list of domains for which it is at all possible to write down a fully explicit Bergman kernel, and it generalizes and simplifies the computations in the special cases mentioned above (see below in Section~\ref{sec-special} for some examples). 
The fact that the kernel is rational is significant in view of the continuing interest in the algebraic nature of the Bergman kernel (see \cite{eben} for some recent results).
Second, the explicit kernel is essential to understanding the regularity of the Bergman projection in norms such as the Sobolev norms (see \cite{EdhMcN20}, where the problem is studied on Generalized Hartogs Triangles starting from the precise form of the kernel). Third, the computation uses combinatorial and algebraic ideas which are of interest in themselves.  We believe that these ideas may be of relevance in the study of Bergman kernels of other domains obtained as quotients, such as the  domains in $\cx^2$ considered in \cite{dallara}, which arise as quotients under  the action of non-Abelian finite reflection groups.

\subsection*{Acknowledgements} This paper reports work done in an 
undergraduate summer research project under the mentorship of Chakrabarti 
and Cinzori, and partially funded by an NSF grant. We thank Central Michigan University for providing space and other resources for this work.  We  also thank Rasha Almughrabi for explaining the details of her computation of the kernel of two-dimensional monomial polyhedra.

\section{A formula for the Bergman kernel of \texorpdfstring{$\Uu_B$}{uub}}
\subsection{Some notation} We introduce, following \cite{summer20}, some extended ``multi-index" type
notation to simplify the writing of our formulas.
\begin{enumerate}[wide]
    \item  For positive integers $m,n$,
    we let $\Z^{m\times n}$ denote the collection of $m\times n$ matrices with integer entries. Similarly,
    $\cx^{m\times n}$ is the space of complex $m\times n$ matrices.     For an $n\times n$ square matrix $A$ with $n\geq 2$, and for $1\leq j,k \leq n$, we denote: 
    \begin{equation}\label{eq-matrixentry}
        [A]^j_k \text{ or } a^{j}_k:= \text{ the entry of $A$ at the intersection of the $j$-th row and the $k$-th column}.
    \end{equation}
      \item For an $n\times n$ matrix $A$, and $1\leq j \leq n$, we denote by $a^j$ the $j$-th row of $A$, so that $a^j$ is a row vector of length $n$. Similarly, $a_j$ is the $j$-th column of $A$ 
    and is therefore a column vector of height $n$. If $A\in \Z^{n\times n}$, using the definition
    \eqref{eq-matrixentry}, we can write for $1\leq j, k \leq n$:
    \begin{equation}
        \label{eq-rowcolumn}a^j =(a^j_1,\dots, a^j_n)\in \Z^{1\times n}, \quad\text{ and } a_k= (a_k^1,\dots, a_k^n)^T\in \Z^{n\times 1}.
    \end{equation}
    \item For a matrix $M\in \cx^{n\times n}$
    , denote by $\adj M\in \cx^{n\times n}$ the \emph{adjugate matrix} of $M$. Recall that, by definition, the entry 
at the $j$-th row and $k$-th column of $\adj M$ is given by
\begin{equation}
    \label{eq-adjdef}
    [\adj M]^j_k=(-1)^{j+k}\det\left(M[k,j]\right), 
\end{equation}
where $M[k,j]$ denotes the $(n-1)\times (n-1)$ submatrix of $M$ obtained by removing the $k$-th row and $j$-th column of $M$. For $M$  invertible,  $ \adj M = \det M\cdot M^{-1}$ by  Cramer's rule.

    \item Notice that according to our convention, $\Z^{1\times n}$
    denotes the collection of integer row vectors of length $n$ and 
    $\Z^{n\times 1}$ denotes the collection of integer column vectors of height $n$.
   The elements of the complex Euclidean space $\cx^n$ are thought of as \emph{column} vectors, i.e., we identify $\cx^n\cong \cx^{n\times 1}.$
     To simplify writing, we write column vectors as transposes of row vectors, where transposition is denoted by a superscript $T$: 
     $ \begin{pmatrix}
         z_1\\ \vdots \\z_n 
     \end{pmatrix}
     = (z_1,\dots, z_n)^{T}.$
   \item We use the standard multi-index power notation: if $z=(z_1,\dots, z_n)^T\in \cx^{n\times 1}$ is an $n\times 1$
    column vector and $\alpha=(\alpha_1,\dots, \alpha_n)\in \cx^{1\times n}$ is a $1\times n$ row vector, we denote
    \begin{equation}
        \label{eq-multiindex1}
         z^\alpha = \prod_{j=1}^n z_j^{\alpha_j}= z_1^{\alpha_1}\dots z_n^{\alpha_n},
    \end{equation}
where each power $z_j^{\alpha_j}$ is assumed to be well-defined and where we use the convention $0^0=1$.  
    \item We denote by $\N$ the collection of nonnegative integers.
    Given a matrix $B\in \Z^{n\times n}$ let $B_+\in \N^{n\times n}$ and $B_-\in \N^{n\times n}$ be matrices given by
    \[ (b_+)^j_k= \max\{b^j_k,0 \}, \quad (b_-)^j_k=\max\{-b^{j}_k,0\}.\]
    More succinctly, $B_+=\max\{B,0\}$ and
    $B_-=\max\{-B,0\}$, where the maxima are taken elementwise and $0$ denotes the $n\times n$ zero matrix. 
    As usual, we let $(b_+)^j, (b_-)^j$ be the rows of $B_+, B_-$, and a similar notation is used for the columns.
\end{enumerate}
\subsection{The function \texorpdfstring{$\dd$}{dd}} Let $k,r$ be integers, with $k\geq 1$. The function $\dd$, introduced in \cite{pjm} (and occurring implicitly in \cite{park, zhang1}), is defined by the relation
\begin{equation}
    \label{eq-Dk-gen}
     \left(\frac{1-x^k}{1-x}\right)^2
     = \sum_{r \in \Z} \dd_k(r)x^r. 
\end{equation}
Since the left-hand side of the above equation is a polynomial, for a fixed $k\geq 1$ the quantity $\dd_k(r)$ vanishes for all negative $r$ and for all but finitely many positive values of $r$. A computation shows that
   \begin{equation}\label{eq-Dk}
     \dd_k(r) = \begin{cases} 
              1+r, & 0\leq r \leq k-1,\\
              2k-(1+r), & k\leq r \leq 2k-2,\\
              0, & r < 0 \text{ or } r > 2k-2.
              \end{cases}
 \end{equation}
\subsection{Two assumptions on \texorpdfstring{$B$}{B}}
It is clearly no loss of generality to assume that the integer matrix $B\in \Z^{n\times n}$ defining the bounded
domain $\Uu_B$ has the following two properties:
\begin{enumerate}[wide]
    \item  The determinant of the defining matrix is positive, i.e., 
\begin{subequations}
  \begin{equation}
         \label{eq-detb1}
         \det B >0.
     \end{equation}
Indeed, we must have $\det B\not=0$ since otherwise $\Uu_B$ is not an open set. Further, if the rows of the matrix $B$ are permuted, the new matrix continues to 
define the same monomial polyhedron, so we may assume \eqref{eq-detb1} without loss of generality. 
\item The $n$ entries of each row of $B$ are relatively prime. We write this as
    \begin{equation}
        \label{eq-gcdbj1} \gcd(b^j)=1, \quad 1\leq j \leq n.
    \end{equation}
   \end{subequations}
Indeed, for $1\leq j\leq n$, if $d$ is a positive integer dividing  each entry of the $j$-th row $b^j$ of the matrix $B$, then dividing each entry $b^j_k$ of this row $b^j$ by $d$ results in a matrix which continues to define the same domain.
Therefore, we may divide each 
row of the defining matrix of a monomial polyhedron by the gcd of that row to obtain a new matrix that defines the same monomial polyhedron and whose rows now satisfy \eqref{eq-gcdbj1}.
\end{enumerate}

\subsection{The main result}In the statement of this result, as well as in the sequel, we denote by $\one$ the $1\times n$ row vector each of whose components is  1:
        \begin{equation}
        \label{eq-one}
        \one = (1,\dots,1) \in \Z^{1\times n}.
    \end{equation}
\begin{thm*}
Assume that the matrix $B$ satisfies \eqref{eq-gcdbj1} and \eqref{eq-detb1}.
 Denoting $t=(t_1,\dots, t_n)^T$ with $t_j=p_j\cdot\ol{q_j}$, the Bergman kernel of $\Uu_B$ is
 \begin{equation}
        \label{eq-kernel2}
        K_{\Uu_B}(p,q)=\frac{1}{\pi^n\cdot {(\det B)^{n-1}}}\cdot \frac{\sum\limits_{\nu\in \N^{1\times n}}C_B(\nu)t^{\nu}}{\prod\limits_{j=1}^n (t^{ (b_-)^j}-t^{(b_+)^j})^2},
    \end{equation}
    where  
        \begin{equation}
            \label{eq-CB}
            C_B(\nu)= \prod\limits_{j=1}^n \dd_{\det B}\left(\left(\nu-2\one B_{-}+\one\right)[\adj B]_j-1\right), \quad \nu\in \Z^{1\times n},
        \end{equation}
        with $[\adj B]_j\in \Z^{n\times 1}$ being the $j$-th column of the adjugate matrix of $B$ and $\dd$ as defined in \eqref{eq-Dk-gen}.
    Further,  we have $C_B(\nu)=0$, except perhaps when $\nu=(\nu_1,\dots, \nu_n)$ satisfies
    \begin{equation}
        \label{eq-nujbounds}
        -1+\xi_j \leq \nu_j \leq 2 \sum_{k=1}^n \abs{b_j^k}-1 -\xi_j, \quad 1\leq j \leq n,
    \end{equation}
        with $\xi_j$ being the ceiling
        \begin{equation}
            \label{eq-xij}\xi_j=\left\lceil\frac{1}{\det B}\cdot\sum\limits_{k=1}^n\abs{b_j^k} \right\rceil.
        \end{equation} 
        The kernel $K_{\Uu_B}$ is a  rational function of the variables $t_1,\dots, t_n$, and the representation \eqref{eq-kernel2} is canonical in the sense that the numerator $\sum_{\nu\in \N^{1\times n}}C_B(\nu)t^{\nu}$ and the denominator $\prod_{j=1}^n (t^{ (b_-)^j}-t^{(b_+)^j})^2$ are polynomials in $\cx[t_1,\dots, t_n]$ without a common  factor. 
\end{thm*}

\textbf{Remark:} Geometrically, the boundary of a nontrivial monomial polyhedron has a non-Lipschitz singularity at the origin, and the rest of the boundary is piecewise Levi-flat
and consists of smooth Levi-flat pieces that meet transversely. Since we know from 
\cite{summer20} that the kernel is rational, the known boundary behavior (see \cite{fu}) of the diagonal kernel $K_{\Uu_B}(z,z)\sim \delta(z)^{-2}$ near smooth Levi-flat boundary points, where $\delta$ is the distance to the boundary, already predicts the form of the denominator. However, it does not seem easy to  deduce information about the behavior of the kernel as $z\to 0$ without actually computing it, and it is this behavior that is of greatest interest in applications to the mapping properties of the Bergman projection.

\section{Preliminaries}
\subsection{Matrix powers of vectors} 
     Let $A\in \cx^{n\times n}$ be an $n\times n$ matrix, and let $z=(z_1,\dots, z_n)^T\in\cx^{n\times 1}$
    be an $n\times 1$ column vector. We define a ``matrix power"  $z^A\in \cx^{n\times 1}$ by the formal expression:
    \begin{equation}
        \label{eq-multiindex2}
        z^A = (z^{a^1},\dots, z^{a^n})^T.
    \end{equation}
    For each $k$, on a domain $U_k\subset \cx$,
    if we  choose for each $1\leq j \leq n$ a local branch of $z_k\mapsto (z_k)^{a^j_k}$ for each entry $a^j_k$ 
    of the column $a_k$ of the matrix $A$, 
we obtain a locally defined holomorphic mapping formally given by $z\mapsto z^A$. This defines a 
holomorphic mapping defined on $U_1\times \dots\times U_n \subset \cx^{n\times 1}$ and taking values in $\cx^{n\time 1}:$
\begin{equation}
    \label{eq-monomialmap}
     \phi_A(z)=z^A.
\end{equation}
If $A\in \N^{n\times n}$ is a matrix of nonnegative integers, then $z^A$ is uniquely defined for all $z\in \cx^{n\times 1}$  and $\phi_A: \cx^{n\times 1}\to \cx^{n\times 1}$ is an entire holomorphic mapping. 
  The following is easily proved (see \cite[Lemma~3.8]{summer20} or \cite[Lemma~4.1]{nagelduke}) and can be thought to be a generalization of the formula $\dfrac{d}{dx}x^n=nx^{n-1}$.
    \begin{prop} \label{prop-powerrule} 
    Let $\phi_A$ be locally defined on some open set of $\cx^{n\times 1}$, as in \eqref{eq-monomialmap}.  We have
        $\det \phi_A'(z)= \det A\cdot z^{\one A-\one}.$
    \end{prop}
\subsection{Monomial maps}
    If it happens that $A\in \Z^{n\times n}$, then $\phi_A$ is a globally defined single-valued map (except for a polar set), known as a \emph{monomial map}. 
    
    To discuss the basic properties of monomial maps, we introduce some more notation. 
    \begin{enumerate}[wide]
        \item  Let 
    \[ \cx^*=\cx\setminus \{0\} \quad \text{ and } \T=\{z\in \cx: \abs{z}=1\},\]
    and note that these are groups under complex multiplication. 
    \item We  let $\exp: \cx^{n\times1} \to (\cx^*)^{n\times 1}$ be the componentwise exponential map
    \[ \exp((z_1,\dots, z_n)^T)= (e^{z_1},\dots, e^{z_n})^T.\]
    \item Given matrices or vectors $z,w$ of the same size, we denote by $z\odot w$ the elementwise
    (or Hadamard-Schur) product of $z$ and $w$, which is therefore a matrix or vector of the same size as $z$ and $w$. For example, if $z,w\in \cx^{n\times 1}$ are column vectors of height $n$,
    then $z\odot w\in \cx^{n\times 1}$ is the column vector of height $n$ whose $j$-th entry is
    $z_j w_j$.
    \end{enumerate}

We now summarize the properties of the monomial map $\phi_A$ (for proof, see \cite{summer20}).
     Recall that a \emph{regular} covering map $\pi:E\to B$ is a covering map where the group $\Gamma$ of
    deck transformations acts transitively on each fiber $\pi^{-1}(x), x\in B$. One can then identify $B$ to the topological quotient 
    $E/\Gamma$, and identify $\pi$ to the quotient map.
    \begin{prop}\label{prop-phiA} Suppose that $A\in \N^{n\times n}$. Then the holomorphic mapping $\phi_A:\cx^n\to \cx^n$ restricts to a regular covering map from $(\cx^*)^n$ to $(\cx^*)^n$ where $\cx^*=\cx\setminus\{0\}$. The deck transformation group $\Gamma$ of
    the regular covering $\phi_A$ is isomorphic to the group \begin{subequations}
    \begin{align} 
             \Gamma &= \{ \gamma\in \T^n: \gamma^A=\one^T\}\label{eq-gammadef1}\\
     &= \{\exp\left(2\pi i A^{-1} \nu\right), \nu \in \Z^{n\times 1}\},\label{eq-gammadef2}
    \end{align}
          \end{subequations}
          where the action of the group $\Gamma$ on $\cx^{n\times 1}$ is given by
          \begin{equation}
              \label{eq-gammaaction}
              (\gamma,z)\mapsto \gamma\odot z, \quad \gamma\in \Gamma, z\in \cx^{n\times 1}.
          \end{equation}
              The order of the group $\Gamma$ is given by
   \begin{equation}
       \label{eq-absgamma} \abs{\Gamma}= \abs{\det A}.
   \end{equation}
    \end{prop}
\subsection{Monomial polyhedra as quotient domains} 
The following representation of monomial polyhedra as quotients was first proved in \cite{summer20}.

\begin{prop}\label{prop:covering-mono-polyh}
Let $B\in\Z^{n\times n}$ be the defining matrix of the domain $\Uu_B$ of \eqref{eq-ubdef}.
\begin{enumerate}
    \item (\cite[Proposition~3.2]{summer20}) The matrix $B$ is invertible, and each entry of $B^{-1}$ is nonnegative.
    \item (\cite[Theorem~3.12]{summer20}) Let  
    \begin{equation}
    \label{eq-A}
    A= \adj B = \det B\cdot B^{-1}\in \N^{n\times n}.
\end{equation}
    Then there exists a product domain 
\begin{equation}\label{eq-Omegaproduct}
\Omega = U_1 \times \dots \times U_n \subset \cx^{n\times 1},
\end{equation}
with each factor $U_j$ either a unit disc $\D=\{\abs{z}<1\}\subset \cx$ or a unit punctured disc 
$\D^*=\{0<\abs{z}<1\}\subset \cx$, such that the monomial map $\phi_{A}: \cx^n \to \cx^n$ of \eqref{eq-monomialmap} restricts to a proper holomorphic  map $\phi_A:\Omega \to \Uu_B$.
This map further restricts to a regular covering map 
\[ \phi_A: \Omega\cap (\cx^*)^{n\times 1} \to \Uu_B\cap (\cx^*)^{n\times 1},\]
whose group of deck transformations is isomorphic to  the group $\Gamma\subset \T^{n\times 1}$ defined in  \eqref{eq-gammadef1} and \eqref{eq-gammadef2}, and the group $\Gamma$ acts on $\Omega$
via the action \eqref{eq-gammaaction}.
\end{enumerate}

\end{prop}
\section{Proof of the main theorem and  formula \texorpdfstring{\eqref{eq-kernel2}}{eqkernel2}}
\subsection{Application of Bell's law}
The following easy-to-verify formulas  will be used without comment: if $z$ is an $n\times 1$ column vector, $\alpha$ is a $1\times n$ row vector, and $P$ and $Q$ are $n\times n$ matrices, we have $(z^P)^\alpha= z^{\alpha P}$ and $(z^P)^Q=z^{QP}$ provided all quantities are well-defined.

We apply Bell's transformation law for the Bergman kernel
under a proper holomorphic map (see \cite{belltransactions}) to the monomial map $\phi_A$ given in part (2) of Proposition~\ref{prop:covering-mono-polyh}. Notice also that 
\begin{equation}
\label{ eq-detA}
\det A = \det(\det B\cdot B^{-1})=(\det B)^n \det (B^{-1})=(\det B)^{n-1}>0.
\end{equation}
    Since $\Omega\subset\cx^{n\times 1}$ is a product of $n$ planar domains, each of which is either the unit disc or the punctured 
    unit disc, and the Bergman kernel of a domain remains unchanged on the removal of an analytic set,  we see that $K_{\Omega}=K_{\D^n}$.  Therefore, we have by Bell's transformation law (with $z\in \Omega$, $q\in \Uu_B$,  each not in the branching loci):
    \begin{equation}
    \label{eq-bell1}
     \det\phi_A'(z)\cdot K_{\Uu_B}(\phi_A(z), q)=\sum_{j=1}^{\det A} K_{\D^n}(z, \Phi_j(q))\cdot \ol{\det \Phi_j'(q)},\quad  \end{equation}
    where  $\{\Phi_j\}_{j=1}^{\det A}$ are the locally defined branches of the inverse to $\phi_A$, of which there are $\abs{\Gamma}=\det A$. Let $C=A^{-1}\in \Q^{n\times n}$.
    For each $1\leq j, k \leq n$, fix a local branch of $q=(q_1,\dots, q_n)\mapsto (q_k)^{c^j_k}$ near each point of $\Uu_B$. This gives us a local branch of 
\begin{equation}
    \label{eq-qAinv}
    q \mapsto q^{A^{-1}}
\end{equation}
near each point of $\Uu_B$. Since we have a regular covering map with deck group $\Gamma$, all the local branches of its inverse are given by 
\begin{equation}
    \label{eq-Phigamma}
    \Phi_\gamma(q)= q^{A^{-1}}\odot \gamma, \quad \gamma\in \Gamma.
\end{equation}
Now thanks to Proposition~\ref{prop-powerrule}, we have $\det \phi'_A(z)= \det A \cdot z^{\one A -\one},$  and
    
\begin{equation}
    \label{eq-detPhi'}
    \det \Phi_\gamma'(q)= \det A^{-1}\cdot q^{\one A^{-1}-\one}\cdot \gamma^{\one},
\end{equation}
where $q^{\one A^{-1}}$ is defined to be $(q^{A^{-1}})^{\one}$ using the branch \eqref{eq-qAinv}.
Inserting these expressions, \eqref{eq-bell1} becomes
\begin{equation}
    \label{eq-bell2} \det A \cdot z^{\one A -\one}\cdot K_{\Uu_B}(z^A, q)=\sum_{\gamma\in \Gamma} K_{\D^n}(z,q^{A^{-1}}\odot \gamma)\cdot \det A^{-1}\cdot \ol{q}^{\one A^{-1}-\one}\cdot \ol{\gamma}^{\one}.\end{equation}
Therefore,
\begin{equation}
    \label{eq-bell3}  K_{\Uu_B}(z^A, q)=\frac{1}{(\det A)^2}\cdot \frac{\ol{q}^{\one A^{-1}-\one}}{z^{\one A -\one}}\sum_{\gamma\in \Gamma} \ol{\gamma}^{\one}\cdot K_{\D^n}(z,q^{A^{-1}}\odot \gamma).\end{equation}
We introduce a change of variables 
 $z=p^{A^{-1}},$ where we use the same branch of the $A^{-1}$-th power as in \eqref{eq-qAinv}. Then $z^A=p$. Further, recalling that 
$t=(p_1\ol{q_1},\dots, p_n\ol{q_n})=p \odot \ol{q}$, we have
\[ \frac{\ol{q}^{\one A^{-1}-\one}}{z^{\one A -\one}}= \frac{\ol{q}^{\one A^{-1}}\cdot \ol{q}^{-\one}}{p^\one\cdot p^{(-\one A^{-1}})}= t^{\one A^{-1}-\one}.\]
Since $\Uu_B$ is a Reinhardt domain,  there is a function $k_2$ such that 
\[ K_{\Uu_B}(p,q)=k_2(p_1\ol{q_1},\dots,p_n\ol{q_n})= k_2(p\odot \ol{q}).\]
We can also write
\[ K_{\D^n}(z,w)=\frac{1}{\pi^n}\cdot \prod_{j=1}^n\frac{1}{(1-z_j\ol{w_j})^2}=k_1(z_1\ol{w_1},\dots, k_n\ol{w_n})=k_1(z\odot \ol{w}), \]
where  $\displaystyle{ k_1(\tau)= \frac{1}{\pi^n}\cdot\prod_{j=1}^n \frac{1}{(1-\tau_j)^2}.}$
Since
\begin{align*}
    z\odot \ol{q^{A^{-1}}\odot \gamma}&=p^{A^{-1}}\odot \ol{q^{A^{-1}}\odot \gamma}= t^{A^{-1}} \odot \ol{\gamma}\\&= (\ol{\gamma_1}t^{c^1},\dots, \ol{\gamma_n}t^{c^n})^T,\quad C=A^{-1},
\end{align*}
in terms of $k_1$ and $k_2$, we can rewrite formula \eqref{eq-bell3} as
\begin{align}
    k_2(t)&=\frac{t^{\one A^{-1}-\one}}{(\det A)^2}\cdot \sum_{\gamma\in \Gamma} \ol{\gamma}^{\one}\cdot k_1(t^{A^{-1}}\odot \ol{\gamma})\label{eq-bell4}\\
   &=\frac{t^{\one A^{-1}-\one}}{\pi^n (\det{A})^2}\cdot
    \sum_{\gamma\in \Gamma}\ol{\gamma}^\one \prod_{j=1}^n \frac{1}{(1-t^{c^j} \ol{\gamma_j})^2}\nonumber\\
    &=\frac{t^{-\one}}{\pi^n (\det{A})^2}\cdot
    \sum_{\gamma\in \Gamma}\prod_{j=1}^n \frac{\ol{\gamma_j}t^{c^j}}{(1-\ol{\gamma_j}t^{c^j} )^2}\label{eq-new1}\\
     &=\frac{t^{-\one}}{\pi^n (\det{A})^2}\cdot L(t^{A^{-1}})\label{eq-new2},
\end{align}
where for a column vector $E=(E_1,\dots, E_n)^T$ of indeterminates we set
\begin{equation}
    \label{eq-L}
    L(E)=\sum_{\gamma\in \Gamma} \prod_{j=1}^n \frac{\ol{\gamma_j} E_j}{(1-\ol{\gamma_j}E_j)^2}.
\end{equation}

\subsection{The group \texorpdfstring{$\Gamma_j$}{Gammaj}}
Consider for $1\leq j \leq n$ the set
\begin{equation}
    \label{eq-Gammajdef}
    \Gamma_j = \{\gamma_j\in \T: (\gamma_1,\dots, \gamma_n)^T\in \Gamma\},
\end{equation}
i.e., the projection of the group $\Gamma\subset \T^n$ onto the $j$-th factor of $\T^n$. Since such a projection is a group homomorphism, $\Gamma_j$ is a finite subgroup of $\T$ and therefore cyclic. We compute its order:
\begin{lem}\label{lem-Gammaj}
  For each $1\leq j \leq n$
\begin{equation}
    \label{eq-claim1}\abs{\Gamma_j}=\det B.
\end{equation}  
\end{lem}
\begin{proof}
    
Writing $C=A^{-1}\in \Q^{n\times n}$ we can rewrite \eqref{eq-gammadef2}
as
\begin{equation*}
\Gamma = \left\{ (e^{2\pi i c^1 \nu}, \dots, e^{2\pi i c^n \nu})^T, \nu \in \Z^{n\times 1}  \right\}.  \end{equation*}
Therefore,  denoting by $[\adj A]^j$ the $j$-th row of the adjugate of $A$, that
\begin{align*}
    \Gamma_j&=\{e^{2\pi i c^j \nu}\in \T:\nu\in \Z^{n\times 1} \} \\
    &= \left\{\exp\left(\frac{2\pi i }{\det A}\cdot  [\adj A]^j\nu\right)\in \T:\nu\in \Z^{n\times 1} \right\}\quad \text{by Cramer's rule}\\
    &= \left\{\exp\left({\frac{2\pi i}{\det A}}\cdot \gcd([\adj A]^j)\cdot  k\right) \in \T:k\in \Z \right\}.
\end{align*}
It follows  that 
the group $\Gamma_j$ is a cyclic group 
of order $\displaystyle{ \frac{{\det A}}{\gcd([\adj A]^j)}}$.
Since 
\begin{align}
    \adj A&=\adj (\adj B)= \det (\adj B)\cdot (\adj B)^{-1}\nonumber= (\det B)^{n-1} (\det B \cdot B^{-1})^{-1}\nonumber\\
    &= (\det B)^{n-2}\cdot B,\label{eq-adjA}
\end{align}
we have
    \[ \gcd([\adj A]^j)=\gcd((\det B)^{n-2}\cdot b^j)=(\det B)^{n-2} \gcd(b^j)=(\det B)^{n-2},\]
    where we use the condition \eqref{eq-gcdbj1}. Therefore,
    \[\abs{\Gamma_j}= \frac{\det A}{\gcd([\adj A]^j)}= \frac{(\det B)^{n-1}}{(\det B)^{n-2}}= \det B.\]\end{proof}

    \subsection{Simplification of \texorpdfstring{$L$}{L}}
    We now turn our attention to the function $L$ of \eqref{eq-new2} and \eqref{eq-L}.
    For each $1\leq j\leq n$, dividing and multiplying by $(1-E_j^{\det B})^2$, we write
    \begin{align}
   \frac{\ol{\gamma_j}E_j}{(1-\ol{\gamma_j}E_j)^2}&= 
   \frac{\ol{\gamma_j}E_j}{(1-E_j^{\det B})^2}\cdot\dfrac{(1-E_j^{\det B})^2}{(1-\ol{\gamma_j}E_j)^2}\nonumber\\
   &=\frac{1}{(1-E_j^{\det B})^2}\cdot \ol{\gamma_j}E_j\sum_{r_j\in \Z}\dd_{\det B}(r_j)\cdot(\ol{\gamma_j}E_j)^{r_j}\label{eq-Dk1}\\
   &= \frac{1}{(1-E_j^{\det B})^2}\cdot \sum_{r_j\in \Z} \dd_{\det B}(r_j)\ol{\gamma_j}^{r_j+1}E_j^{r_j+1},\nonumber
\end{align}
where  in \eqref{eq-Dk1} we use \eqref{eq-Dk-gen} with $x=\ol{\gamma_j} E_j$, remembering that $x^{\det B}=E_j^{\det B}$ since $(\ol{\gamma_j})^{\det B}=1$, as $\ol{\gamma_j}\in \Gamma_j$ and the group $\Gamma_j$ has order $\det B$.
 Therefore,
\begin{align}
    L(E)&= \sum_{\gamma\in \Gamma} \prod_{j=1}^n \frac{\ol{\gamma_j} E_j}{(1-\ol{\gamma_j}E_j)^2}\nonumber\\
    &= \sum_{\gamma\in \Gamma} \prod_{j=1}^n\frac{1}{(1-E_j^{\det B})^2}\cdot \sum_{r_j\in \Z} \dd_{\det B}(r_j)\ol{\gamma_j}^{r_j+1}E_j^{r_j+1}\nonumber\\
    &= \frac{\sum_{\gamma\in \Gamma} \prod_{j=1}^n\sum_{r_j\in \Z} \dd_{\det B}(r_j)\ol{\gamma_j}^{r_j+1}E_j^{r_j+1}}{\prod_{j=1}^n(1-E_j^{\det B})^2}\nonumber\\
    &=  \frac{\Lambda(E)}{\Delta(E)}, \label{eq-LLambdaDelta}
\end{align}
   where $\Delta$ and $\Lambda$ are the polynomials in the indeterminates $E_1,\dots, E_n$ given by
   \begin{equation}
       \label{eq-DeltaE}
       \Delta(E)= \prod_{j=1}^n(1-E_j^{\det B})^2
   \end{equation}
        and
        \begin{align}
            \Lambda(E)&=\sum_{\gamma\in \Gamma}\prod_{j=1}^n\sum_{r_j\in \Z} \dd_{\det B}(r_j) \ol{\gamma_j}^{r_j+1}E_j^{r_j+1}\nonumber\\
            &= \sum_{\theta\in \Z^{1\times n}} \left(\sum_{\gamma\in \Gamma} \prod_{j=1}^n \dd_{\det B}(\theta_j-1)\ol{\gamma_j}^{\theta_j} \right) E^\theta\label{eq-LambdaStandard}\\
            &= \sum_{\theta\in \Z^{1\times n}}\Lambda_\theta E^\theta, \label{eq-LambdaE}
        \end{align}
where in \eqref{eq-LambdaStandard} we have gathered all coefficients associated to each monomial $E^\theta$ to put $\Lambda(E)$ in the standard form, and consequently,
\begin{equation}
            \label{eq-lambdatheta}\Lambda_\theta= \sum_{\gamma\in \Gamma}\ol{\gamma}^\theta\prod_{j=1}^n \dd_{\det B}(\theta_j-1). 
        \end{equation}     
To simplify \eqref{eq-LambdaE}, we notice that for $\alpha\in \Gamma$, we have (recall that $\odot$ stands for entrywise multiplication of vectors)
 \begin{align}
            L(\alpha \odot E)&= \sum_{\gamma\in \Gamma} \prod_{j=1}^n \frac{\alpha_j\ol{\gamma_j} E_j}{(1-\alpha_j\ol{\gamma_j}E_j)^2} = \sum_{\beta\in \Gamma} \prod_{j=1}^n \frac{\ol{\beta_j} E_j}{(1-\ol{\beta_j}E_j)^2}\label{eq-thetaj}\\
            &= L(E),\nonumber\end{align}
where in the last expression in \eqref{eq-thetaj}, we set $\beta = \ol{\alpha}\odot \gamma$, i.e., $\beta_j= \ol{\alpha_j}\gamma_j$,  and reindex the sum over the group $\Gamma$.
    Also,
\[    \Delta(\alpha\odot E) = \prod_{j=1}^n (1-(\alpha_jE_j)^{\det B})^2 
    = \prod_{j=1}^n (1-\alpha_j^{\det B} \cdot E_j^{\det B})^2 
    = \prod_{j=1}^n (1-E_j^{\det B})^2 
    = \Delta(E).
\]
Since $\Lambda(E)=L(E)\cdot \Delta(E)$ it follows that for each $\alpha\in \Gamma$,
\[ \Lambda(\alpha\odot E)= L(\alpha\odot E)\cdot \Delta(\alpha\odot E)= \Lambda(E). \]
Using the representation \eqref{eq-LambdaE} of $\Lambda(E)$, this is equivalent to the fact that for each $\alpha\in \Gamma$,
\[ \sum_{\theta\in \Z^{1\times n}} \Lambda_\theta\alpha^\theta E^\theta = \sum_{\theta\in \Z^{1\times n}} \Lambda_\theta E^\theta, \quad \text{ i.e.,} \quad   \sum_{\theta\in \Z^{1\times n}} \Lambda_\theta\cdot(\alpha^\theta-1) E^\theta =0.\]
Therefore, for  $\theta\in \Z^{1\times n}$, we can have $\Lambda_\theta\not=0$ only if 
$\alpha^\theta=1$ for each $\alpha\in \Gamma$, i.e., using the representation \eqref{eq-gammadef2} for the group $\Gamma$, for each $\nu\in \Z^{n\times 1}$, we have
\[1= \left(\exp(2\pi i A^{-1}\nu) \right)^\theta= e^{2\pi i \theta A^{-1}\nu},\]
which is to say that 
\begin{equation}
    \label{eq-nu}
    \theta A^{-1}\nu \in \Z \quad \text{ for each } \nu \in \Z^{n\times 1}.
\end{equation}
We claim that \eqref{eq-nu} holds if and only if 
\begin{equation}
    \label{eq-m}
    \theta=mA \quad\text{for an } m\in \Z^{1\times n}.
\end{equation}
Indeed, if $\theta=mA$ for an integer row vector $m$, then for each integer column vector 
$\nu$:
\[ \theta A^{-1}\nu= (mA)A^{-1}\nu= m\nu\in \Z.\]
Conversely, suppose that for each $ \nu \in \Z^{n\times 1}$, we have $\theta A^{-1}\nu\in \Z$.
For $1\leq j \leq n$, let $e_j$ denote the $j$-th standard basis vector in $\Z^{n\times 1}$, i.e., $e_j$ is a column vector with $n$ entries, of which the $j$-th entry is 1 and the others are zeroes, and set $m_j=\theta A^{-1}e_j,$ so that $m_j\in \Z$ by hypothesis. But then $m_j$ is the $j$-th entry of the row vector $\theta A^{-1}$ and therefore this vector is in $\Z^{1\times n}.$ Setting $m=\theta A^{-1}$, we have $\theta= mA,$ as needed.

Therefore, $\Lambda_\theta\not=0$ for a $\theta\in \Z^{1\times n}$ if and only if \eqref{eq-m} holds, and consequently, the expression \eqref{eq-LambdaE} simplifies to
\begin{equation}
    \label{eq-LambdaEsimp}
   \Lambda(E)= \sum_{\substack{\theta=mA\\m\in \Z^{1\times n}}}\Lambda_\theta E^\theta = \sum_{m\in \Z^{1\times n}} \Lambda_{mA}E^{mA}=  \sum_{m\in \Z^{1\times n}} \Lambda_{mA}\cdot(E^{A})^m. 
\end{equation}
Using the representation \eqref{eq-lambdatheta} of $\Lambda_\theta$, we see that
\begin{align}
  \Lambda_{mA}&=\sum_{\gamma\in \Gamma} \ol{\gamma}^{mA} \prod_{j=1}^n \dd_{\det B}(ma_j-1)
  \label{eq-lma1}\\
 &= \sum_{\gamma\in \Gamma} (\ol{\gamma}^{A})^m \prod_{j=1}^n \dd_{\det B}(ma_j-1)
 \nonumber\\
 &= \sum_{\gamma\in \Gamma} (\one^T)^m \prod_{j=1}^n \dd_{\det B}(ma_j-1)
 \label{eq-lma2}\\
 &= \abs{\Gamma} \prod_{j=1}^n \dd_{\det B}(ma_j-1)
 \nonumber\\
&= {\det A}\cdot \prod_{j=1}^n \dd_{\det B}(ma_j-1),\label{eq-lma3}
\end{align}
where in \eqref{eq-lma1}, we use the fact that the $j$-th entry of the row vector
$mA\in \Z^{1\times n}$ is $ma_j\in \Z$, where $a_j$ denotes the $j$-th column of the matrix  $A=\adj B$. In \eqref{eq-lma2} we use the characterization \eqref{eq-gammadef1} of the group $\Gamma$, and finally, in \eqref{eq-lma3} we use the fact that $\Gamma$ has $\abs{\det A}$ elements. Therefore, using \eqref{eq-LambdaEsimp}, we have
\begin{align}
\Lambda(t^{A^{-1}})&=\sum_{m\in \Z^{1\times n}}\Lambda_{mA}\cdot((t^{A^{-1}})^A)^m \nonumber\\
&={\det A}\cdot\sum_{m\in \Z^{1\times n}} \prod_{j=1}^n \dd_{\det B}(ma_j-1) t^m.\label{eq-LambdatAinv}
\end{align}
We also have, using \eqref{eq-DeltaE},
$\displaystyle{ \Delta(t^{A^{-1}})= \prod_{j=1}^n\left(1- (t^{c^j})^{\det B} \right)^2= \prod_{j=1}^n\left(1- t^{\det B\cdot c^j} \right)^2,}$
where $c^j$ denotes the $j$-th row of the matrix $C=A^{-1}\in \Q^{n\times n}$. Notice that
\[ C=A^{-1}=(\adj B)^{-1}=(\det B\cdot B^{-1})^{-1}=\frac{1}{\det B}\cdot B,\]
so $\displaystyle{ \det B\cdot c^j=   \det B\cdot \frac{1}{\det B}\cdot b^j= b^j.}$
Therefore, we obtain
\begin{equation}
    \label{eq-DeltatAinv}
     \Delta(t^{A^{-1}})= \prod_{j=1}^n\left(1- t^{{b^j}} \right)^2.
\end{equation}
\subsection{An intermediate expression for the kernel}
Since by definition \eqref{eq-LLambdaDelta} we have $L= \frac{\Lambda}{\Delta},$ using 
the representations \eqref{eq-LambdatAinv} and \eqref{eq-DeltatAinv} in \eqref{eq-new2}, we obtain
\begin{align*}
    k_2(t)&=\frac{t^{-\one}}{\pi^n (\det{A})^2}\cdot L(t^{A^{-1}})= \frac{t^{-\one}}{\pi^n (\det{A})^2}\frac{\Lambda(t^{A^{-1}})}{\Delta(t^{A^{-1}})}\\
    &= \frac{t^{-\one}}{\pi^n (\det{A})^2}\cdot \frac{{\det A}\cdot\sum_{m\in \Z^{1\times n}} \prod_{j=1}^n \dd_{\det B}(ma_j-1) t^m}{\prod_{j=1}^n\left(1- t^{{b^j}} \right)^2}\\
    &= \frac{1}{\pi^n\cdot {\det{A}}}\cdot\frac{\sum_{m\in \Z^{1\times n}} \prod_{j=1}^n \dd_{\det B}(ma_j-1) t^{m-\one}}{\prod_{j=1}^n\left(1- t^{{b^j}} \right)^2},
\end{align*}
which taking into account \eqref{ eq-detA}, shows that
\begin{equation}
    K_{\Uu_B}(p,q)= \frac{1}{\pi^n\cdot {(\det B)^{n-1}}}\cdot \frac{P(t)}{Q(t)}, \label{eq-kernel1}
    \end{equation}
    where, recalling that $A=\adj B$,
    \begin{equation}\label{eq-P}
        P(t)= \sum\limits_{m\in \Z^{1\times n}}\left(\prod\limits_{j=1}^n \dd_{\det B}\left(m\,[\adj B]_j-1\right)\right)t^{m-\one},
    \end{equation}
    and
    \begin{equation}
        \label{eq-Q}
        Q(t)=\prod\limits_{j=1}^n (1-t^{b^j})^2.
\end{equation}
    Notice that $Q$ is a Laurent polynomial in the $n$ variables $t_1,\dots, t_n$ with integer coefficients and therefore $Q$  a rational function. Also, $P$ is a Laurent series in these $n$ variables. In fact, a more careful book-keeping shows that $P$ is also a Laurent polynomial with integer coefficients.  We will show that  $P/Q$ is in fact a rational function by representing it as the ratio of two polynomials. 
\subsection{Reduction  to ratio of polynomials}
Recall that $B_+=\max\{B,0\}$, and $B_-=\max\{-B,0\}$, where maxima of matrices are taken entrywise, and consequently, $B=B_+-B_-$. 
   Notice that the quantity $t^{2\one B_{-}}$ has two alternate representations:
   \begin{subequations}
    \begin{align}
        t^{2\one B_{-}}&=t^{2 \sum_{j=1}^n(b_{-})^j}
        =\prod_{j=1}^n t^{2(b_{-})^j}\label{eq-t21-1}\\
        &=  \prod_{j=1}^n t_j^{2 \one(b_{-})_j},\label{eq-t21-2}
    \end{align}
   \end{subequations}
   where in \eqref{eq-t21-2}, we have used that $\one B_{-}=(\one(b_{-})_1, \dots, \one(b_{-})_n).$
We multiply both the numerator and the denominator of \eqref{eq-kernel1}
by the monomial $t^{2\one B_{-}}$. For the denominator, using representation \eqref{eq-t21-1}, we obtain
\begin{align}t^{2\one B_{-}}\cdot Q(t)&=
    t^{2\one B_{-}}\cdot\prod\limits_{j=1}^n (1-t^{b^j})^2= \prod_{j=1}^n t^{2(b_{-})^j}\cdot\prod\limits_{j=1}^n (1-t^{b^j})^2\nonumber= \prod_{j=1}^n (t^{(b_{-})^j} - t^{(b_{-})^j+b^j})^2\nonumber\\
     &= \prod_{j=1}^n (t^{(b_{-})^j} - t^{(b_{+})^j})^2,\label{eq-t21b-deno}
\end{align}
where in the last line we have used the fact that $b^j= (b_{+})^j-(b_{-})^j$. Now multiplying the numerator of \eqref{eq-kernel1} by $t^{2\one B_{-}}$ and using the representation \eqref{eq-t21-2}, we obtain:
\begin{align}t^{2\one B_{-}}\cdot P(t)
    =& t^{2\one B_{-}}\cdot \sum\limits_{m\in \Z^{1\times n}}\left(\prod\limits_{j=1}^n \dd_{\det B}\left(m\,[\adj B]_j-1\right)\right)t^{m-\one}\nonumber\\
    =& \sum\limits_{m\in \Z^{1\times n}}\prod_{j=1}^n t_j^{2 \one(b_{-})_j} \cdot\prod\limits_{j=1}^n \dd_{\det B}\left(m\,[\adj B]_j-1\right)\cdot \prod_{j=1}^n t_j^{m_j-1}\nonumber\\
    =& \sum\limits_{m\in \Z^{1\times n}}\prod_{j=1}^n\dd_{\det B}\left(m\,[\adj B]_j-1\right)t_j^{m_j+2 \one(b_{-})_j-1}.\label{eq-t21b-num}
\end{align}
 In the outer sum of \eqref{eq-t21b-num}, we reindex using a new summation index $\nu =(\nu_1,\dots, \nu_n)\in \Z^{1\times n}$ by setting 
 $ \nu_j= m_j+2 \one(b_{-})_j-1$ for $1\leq j \leq n$,
 or in vector notation,
 \[ \nu= m+2\one B_{-}-\one.\]
 Notice that the mapping from $m$ to $\nu$ is nothing but a translation 
 in the integer lattice $\Z^{1\times n}$, so therefore $\nu$ can be used as an index of summation instead of $m$ in \eqref{eq-t21b-num}. Solving for $m$ we obtain
 \begin{equation}
     \label{eq-mnu}
     m=\nu- 2\one B_{-}+\one,
 \end{equation}
 or, written in terms of the components,
 $ m_j = \nu_j - 2 \one(b_{-})_j+1, \quad 1\leq j \leq n.$
Reindexing, we obtain
\begin{align}
    \eqref{eq-t21b-num}&= 
    \sum\limits_{\nu\in \Z^{1\times n}}\prod_{j=1}^n\dd_{\det B}\left((\nu- 2\one B_{-}+\one)[\adj B]_j-1\right)t^{\nu}\nonumber\\
    &= \sum\limits_{\nu\in \Z^{1\times n}}C_B(\nu)t^\nu,\label{eq-t21b-num2}
\end{align}
with $C_B$ as in \eqref{eq-CB}. Therefore, using  \eqref{eq-t21b-num2} and \eqref{eq-t21b-deno} in \eqref{eq-kernel1}, we see that

\begin{align*}
     K_{\Uu_B}(p,q)&= \frac{1}{\pi^n\cdot {(\det B)^{n-1}}}\cdot \frac{t^{2\one B_{-}}\cdot P(t)}{t^{2\one B_{-}}\cdot Q(t)}\\
     &= \frac{1}{\pi^n\cdot {(\det B)^{n-1}}}\cdot\frac{\sum\limits_{\nu\in \Z^{1\times n}}C_B(\nu)t^\nu}{\prod_{j=1}^n (t^{(b_{-})^j} - t^{(b_{+})^j})^2},
     \end{align*}
     which is the claimed formula \eqref{eq-kernel2}, except the sum in the numerator is over $\nu\in \Z^{1\times n}$ rather than the finite set given by \eqref{eq-nujbounds}.
\subsection{Bounds on \texorpdfstring{$\nu_j$}{nuj}}
     
     To complete the proof, we need 
     to show that if $\nu$ does not satisfy the conditions \eqref{eq-nujbounds}, then we have $C_B(\nu)=0$.  Notice that the quantity 
     $\xi_j$ in \eqref{eq-xij} is an integer greater than or equal to 1, 
     since $\xi_j$ is the ceiling of a positive number.  Therefore, if \eqref{eq-nujbounds} has been established, then from the left inequality we would know that $\nu_j\geq 0$ for each $j$, and, consequently, the sum in the numerator will only have $\nu\in \N^{1\times n}$, i.e., the numerator is a polynomial.

     From  \eqref{eq-Dk}, we see that $\dd_k(r)=0$ for those $r$ which do not satisfy
     $ 0 \leq r \leq 2k-2.$
     Therefore, $\dd_{\det B}\left((\nu- 2\one B_{-}+\one)[\adj B]_j-1\right)=0$ except when we have 
     \begin{equation}
         \label{eq-bound1}
          0\leq (\nu- 2\one B_{-}+\one)a_j-1 \leq 2\det B-2,
     \end{equation}
     using the notation $A=\adj B$ introduced above in \eqref{eq-A}, and $a_j$ being the $j$-th column of $A$.
    From the definition \eqref{eq-CB} of the coefficients $C_B(\nu)$, it follows that $C_B(\nu)=0$ provided $\nu$ does not satisfy \eqref{eq-bound1} for at least one $j$ with $1\leq j \leq n$. 
    To manipulate the system of inequalities \eqref{eq-bound1} in an efficient manner, we use 
    the elementwise inequality notation defined as follows: if $P,Q\in \rl^{m\times n}$ are real matrices or vectors
    of the same size, then 
    \[ P\preceq Q \quad \text{ means } p^j_k \leq q^j_k, \quad 1\leq j\leq m \text{ and } 1 \leq k \leq n.\]
    We will also write $Q\succeq P$ to mean $P \preceq Q$ if convenient. 
    Using this notation, we can say that 
    the set of $\nu$ for which $C_B(\nu)\not=0$ is contained in the set of $\nu\in \Z^{1\times n}$ given by 
    \begin{equation}
  \label{eq-bound2}
  0\preceq (\nu- 2\one B_{-}+\one)A-\one \preceq (2\det B-2)\one.
    \end{equation}
  which can be rearranged to read
  \begin{equation} \label{eq-rearrange}
      \alpha \preceq \nu A  \preceq \beta,
  \end{equation}
       where $\alpha,\beta\in \Z^{1\times n}$ are given by 
       \begin{subequations}
           \begin{equation}\label{eq-alpha}
               \alpha=( 2\one B_{-}-\one)A +\one
           \end{equation}
           and
           \begin{equation}
               \label{eq-beta}
               \beta= (2\det B-1)\one + ( 2\one B_{-}-\one)A.
           \end{equation}
       \end{subequations}
       Notice that since $B_{-}\succeq 0$ and $B_{+}\succeq 0$ by definition, we have from \eqref{eq-rearrange}:
       \begin{subequations}
           \begin{equation}
               \label{eq-b+}
               \alpha B_{+}\preceq \nu A B_{+} \preceq \beta B_{+}
           \end{equation}
           and
       \begin{equation}
           \label{eq-negb-}
           -\beta B_{-}\preceq - \nu A B_{-}\preceq - \alpha B_{-}.
       \end{equation}\end{subequations}
Adding \eqref{eq-b+} and \eqref{eq-negb-} we obtain
\begin{equation}
    \label{eq-add}
    \alpha B_{+}- \beta B_{-} \preceq \nu A (B_{+}- B_{-})  \preceq \beta B_{+}-\alpha B_{-}.
\end{equation}
To simplify \eqref{eq-add}, first note that since $B=B_{+}- B_{-}$, and 
$A=\adj B = \det B\cdot B^{-1},$ we have for the middle term
\[ \nu A (B_{+}- B_{-})= \nu AB = \nu\cdot(\det B \cdot I)= \det B\cdot \nu.\]
Let us denote $ \abs{B}=B_{+}+B_{-},$
so that the entry at the $j$-th row and $k$-th column of $\abs{B}$ is $\abs{b^j_k}.$
For the leftmost term of \eqref{eq-add},  we have
\begin{align*}
     \alpha B_{+}- \beta B_{-} &= \left( ( 2\one B_{-}-\one)A +\one\right) B_{+}- \left((2\det B-1)\one + ( 2\one B_{-}-\one)A \right)  B_{-}\\
     &=  2\one B_{-}A  B_{+}-\one AB_{+}+\one B_{+}- 2\det B\cdot \one B_{-}+\one B_{-}- 2\one B_{-}AB_{-}+ \one A B_{-}\\
     &=  2\one B_{-}A  (B_{+}-B_{-})-\one A(B_{+}-B_{-})+\one (B_{+}+B_{-})- 2\det B\cdot \one B_{-}\\
     &= 2\one B_{-}A B-\one AB+ \one \abs{B}- 2\det B\cdot \one B_{-}\\
     &= 2\det B\cdot\one B_{-}-\det B\cdot \one + \one \abs{B}- 2\det B\cdot \one B_{-}\\
     &= \one \abs{B} -\det B\cdot \one.
\end{align*}
For the rightmost term in \eqref{eq-add} we have
\begin{align*}
\beta B_{+}-\alpha B_{-}&=\left( (2\det B-1)\one + ( 2\one B_{-}-\one)A\right) B_{+}-\left(( 2\one B_{-}-\one)A +\one \right) B_{-}\\
&=2\det B\cdot\one B_{+}-\one B_{+} + 2\one B_{-}A B_{+}-\one A B_{+}
- 2\one B_{-}A B_{-}+\one AB_{-}-\one B_{-}\\
&=2\det B\cdot\one B_{+}-\one (B_{+}+B_{-}) + 2\one B_{-}A (B_{+}-B_{-})-\one A (B_{+}-B_{-})\\
&=2\det B\cdot\one B_{+}-\one  \abs{B} + 2\one B_{-}A B-\one A B\\
&=2\det B\cdot\one B_{+}-\one  \abs{B}+2\det B\cdot\one B_{-}-\det B\cdot\one\\
&= 2\det B\cdot\one(B_{+}+ B_{-}) -\one  \abs{B}-\det B\cdot\one\\
&=(2\det B-1)\cdot\one \abs{B}-\det B\cdot\one.
\end{align*}
Putting together the three obtained expressions, \eqref{eq-add} becomes
\[ \one \abs{B} -\det B\cdot \one\preceq \det B\cdot \nu \preceq(2\det B-1)\cdot \one\abs{B}-\det B\cdot\one. \]
Since $\det B>0$ by hypothesis, this is equivalent to
\[ -\one + \frac{1}{\det B}\cdot \one \abs{B} \preceq \nu \preceq-\one+ \left( 2 -\frac{1}{\det B}\right)\cdot\one \abs{B}. \]
In terms of  components this takes the form
\begin{equation}\label{eq-bound4}
    -1 + \frac{1}{\det B}\cdot \sum_{k=1}^n \abs{b^k_j} \leq \nu_j \leq - 1 + \left( 2 -\frac{1}{\det B}\right)\cdot \sum_{k=1}^n \abs{b^k_j}, \quad 1\leq j \leq n. 
\end{equation}
Since  $\nu_j$ is an integer, we can replace the leftmost member by its ceiling, and the rightmost member by its floor to obtain valid inequalities in terms of integers. Recalling the definition \eqref{eq-xij} of $\xi_j$, notice that the ceiling of the leftmost member is
\begin{subequations}
    \begin{equation}
    \label{eq-left}  \left\lceil -1 + \frac{1}{\det B}\cdot \sum_{k=1}^n \abs{b^k_j}\right\rceil=-1+\xi_j,
\end{equation}
  and the floor of the rightmost member is
\begin{align}
\left\lfloor  - 1 + \left( 2 -\frac{1}{\det B}\right)\cdot \sum_{k=1}^n \abs{b^k_j}\right\rfloor&= -1 +  2  \sum_{k=1}^n \abs{b^k_j}
+ \left \lfloor - \frac{1}{\det B}\cdot \sum_{k=1}^n \abs{b^k_j}\right\rfloor\nonumber\\
&= -1 +  2  \sum_{k=1}^n \abs{b^k_j}- \xi_j,\label{eq-right}
\end{align}\end{subequations}
where we use the fact that $\lfloor -x \rfloor = - \lceil x \rceil.$
Plugging in \eqref{eq-left} and \eqref{eq-right} into \eqref{eq-bound4},
the  inequalities \eqref{eq-nujbounds} follow.

\subsection{Canonicity of the representation} To complete the proof of the main theorem, we will now show that the representation \eqref{eq-kernel2} as the ratio of two polynomials is canonical, in the sense that there is no common irreducible factor of the numerator and denominator. 

To see this, first, recall (see \cite{fu94}) that if $\Omega\subset \cx^n$ is a bounded pseudoconvex 
domain and $p\in b\Omega$ is a boundary point in a neighborhood of which $b\Omega$ is smooth, then there is a neighborhood $U$ of $p$ and a constant $C>0$ such that 
\begin{equation}
    \label{eq-fu}
    K_\Omega(z,z) \geq \frac{C}{\delta(z)^2} \quad\text{ for each } z\in U\cap \Omega,
\end{equation}
where $\delta(z)$ denotes the distance from the point $z$ to $b\Omega$. This is clear in the case $n=1$, by comparing $K_\Omega(z,z)$ with the diagonal Bergman kernel of a disk in $\Omega$ tangent to $b\Omega$, and the general case follows by an inductive argument, using the famous Ohsawa-Takegoshi $L^2$-holomorphic extension theorem in the induction step. 

To complete the proof, we use the following algebraic fact whose proof is postponed to the end of the section:
\begin{lem}\label{lem-irreducible}
    For $1\leq k \leq n$, the polynomial $p_k(t)=t^{(b_-)^k} - t^{( b_+)^k}$ is irreducible in 
    $\cx[t_1,\dots, t_n]$.
\end{lem}
Assuming the lemma for the moment, suppose that the rational function $K_{\Uu_B}$ is \textit{not} in canonical form. Then, there is a $j$ with $1\leq j \leq n$ such that $p_j$ divides the numerator
$\sum_{\nu\in \N^{1\times n}}C_B(\nu)t^{\nu}$ in the ring $\cx[t_1,\dots, t_n]$. We can remove this common factor, leaving us with a denominator containing $p_j$ to at most the first power. Now, let $q\not=0$ be a point on the face $F_j=\{p_j(\abs{z_1}^2,\dots, \abs{z_n}^2)=0\}\subset b\Uu_B$ which does not belong to any other face $F_k, k\not=j$. Notice that the polynomial function $r(z)= p_j(\abs{z_1}^2,\dots, \abs{z_n}^2)$ is a defining function of the domain $\Uu_B$ near the smooth boundary point $q$, and so, since $p_j(t)$ occurs in the denominator at most to the first power,  from \eqref{eq-kernel2}, there is a neighborhood $V$ of $q$
and a constant $C'>0$ such that 
  \begin{equation}
      \label{eq-bound3} K_{\Uu_B}(z,z) \leq \frac{C'}{r(z)},\quad  z\in V\cap \Uu_B.
  \end{equation}
However, shrinking $V$ if needed, we have that $r$ is comparable to $\delta$, the distance to the boundary, so using \eqref{eq-fu} we have that there is a constant $C''>0$ such that 
\[  K_{\Uu_B}(z,z)\geq \frac{C''}{r(z)^2},\quad  z\in V\cap \Uu_B,\]
which contradicts \eqref{eq-bound3} for $z$ close enough to $q$, thus showing that $p_j$ is not a factor of the numerator, and therefore the representation \eqref{eq-kernel2} is already in its lowest terms.

\begin{proof}[Proof of Lemma~\ref{lem-irreducible}]
Denote $\alpha=(b_-)^k, \beta=(b_+)^k\in \N^{1\times n}$
so that $b^k=\beta - \alpha$ and  $p_k(t)=t^\alpha-t^\beta= t_1^{\alpha_1}\cdots t_n^{\alpha_n}- 
t_1^{\beta_1}\cdots t_n^{\beta_n}$. Since $\det B\not=0$, each row of $B$ has at least one 
nonzero entry and at least one of $\alpha, \beta$ is nonzero. Therefore, renaming $t_1,\dots, t_n$, we can assume that
$\alpha_1\not=0$. Further, the assumption  \eqref{eq-gcdbj1} means that 
\[ \gcd\{\alpha_1,\dots, \alpha_n, \beta_1,\dots, \beta_n\}=1.\]
Also, by definition of $B_+, B_-$, we have the property
\begin{equation}
    \label{eq-prop3}
    \alpha_j\not=0 \Rightarrow \beta_j=0, \text{ and } \beta_j \not=0\Rightarrow \alpha_j=0.
\end{equation}
Let $F$ denote the rational function field $\cx(t_2,\dots, t_n)$, the field of fractions of the UFD $R=\cx[t_2,\dots, t_n]$.
Identifying $\cx[t_1,\dots,t_n]$ with $R[t_1]$ it follows by Gauss's Lemma
that the polynomial $p_k$ in $R[t_1]$ is irreducible in $R[t_1]$ provided
it is primitive in $R[t_1]$ and irreducible in $F[t_1]$. Notice that since 
$\alpha_1\not =0$, it follows that $p_k$ is a polynomial of degree $\alpha_1 \geq 1$ as an element of either $R[t_1]$ or $F[t_1].$ 

Recall that a polynomial in $R[t_1]$ is primitive if the gcd of its coefficients in $R$ is 1. Since $\alpha_1\not=0$, it follows that $\beta_1=0$   by \eqref{eq-prop3},  and we can write $ p_k(t) = Pt_1^{\alpha_1}-Q$. We can factor the coefficients  into 
irreducibles of $R$ as
$Q= t_2^{\beta_2}\dots t_n^{\beta_n}\in R$  and 
$P=t_2^{\alpha_2}\dots t_n^{\alpha_n}\in R$. Since 
 $t_2,\dots, t_n$ are irreducible in the UFD $R$, it follows from these factorizations that  $\gcd(P,-Q)=1$, and therefore $p_k$ is primitive in $R[t_1]$.

 In $F[t_1]$, we can write $p_k(t)= P(t_1^{\alpha_1}-  P^{-1}\cdot Q)$, so that we only need to show that 
 the polynomial $q(t_1)=t_1^{\alpha_1}-P^{-1}\cdot Q$ is irreducible in $F[t_1]$. This is obvious if $\alpha_1=1$, so assume that $\alpha_1\geq 2$.
 We  now claim that $P^{-1}\cdot Q\in F$
 is not a $d$-th power in $F$ for any $d\geq 2$ that divides $\alpha_1$.
 Indeed,  we have a factorization of $P^{-1}\cdot Q$
 into irreducible elements of $R$ 
 \[P^{-1}\cdot Q= t_2^{\beta_2-\alpha_2}\dots t_n^{\beta_n-\alpha_n} =\prod_{j=2}^n t_j^{b^k_j},\]
and since $\gcd(b^k)=1$, it follows that if $d\geq 2$ is 
a divisor of $\alpha_1= b^k_1$, then there is at least one $2\leq j \leq n$ such that $d$ does not divide $b^k_j$. This means that $P^{-1}\cdot Q$ is not a $d$-th power in $F$, since one of its irreducible factors $t_j$ is raised to a power $b^k_j$ not divisible by $d$.

 Now let $c$ be an $\alpha_1$-th root of $P^{-1}\cdot Q$ in an extension field.  Since by the above claim, $P^{-1}\cdot Q$ is not an $\alpha_1$-th root, it follows that $c\not \in F$. More generally, 
  we have
 \begin{equation}
     \label{eq-cs} c^s \not \in F, \quad 1 \leq s < \alpha_1.
 \end{equation}
We already know this for $s=1$, so consider the smallest $s$ with $2\leq s <\alpha_1$ for which $c^s\in F$. We claim that $s$ divides $\alpha_1$. 
 Let $s'$ be the remainder in dividing $\alpha_1$ by $s$, so that $c^{s'}=(c^{s})^{-T}\cdot c^{\alpha_1}\in F$ for some positive integer $T$, which contradicts the minimality of $s$, since $s'\not=1$. But then $c^s$ is an $\frac{\alpha_1}{s}$-th root of $P^{-1}\cdot Q$, contradicting the fact that  $P^{-1}\cdot Q$ is not a $d$-th power in $F$ for each divisor $d$ of $\alpha_1$.
  
 In the ring $F(c)[t_1]$, the polynomial $q(t_1)$ splits completely as
 $\displaystyle{t_1^{\alpha_1}-P^{-1}\cdot Q= \prod_{j=1}^{\alpha_1}(t_1-c\,\omega^j)}$,
 where $\omega=e^{\frac{2\pi i}{\alpha_1}}\in \cx$.
If $q(t_1)$ has a nontrivial divisor $A(t_1)$ in $F[t_1]$, there is a subset $\emptyset \not =S\subsetneq \{1,\dots, n\}$ and a $\lambda\in F$ such that   $A(t_1)= \lambda\prod_{j\in S}(t-c\cdot\omega^j)$. Looking at the constant term of $A$ we see that 
\[ \lambda \prod_{j\in S}c\omega^j=\lambda \cdot c^{\abs{S}}\cdot \omega^{\sum_{j \in S}j}\in F.\] Since $\lambda\in F$ and $\omega^{\sum_{j \in S}j}\in \cx\subset F$, it follows that $c^{\abs{S}}\in F$. But this contradicts \eqref{eq-cs}, showing that $q(t_1)$ is irreducible in $F[t_1]$
and completing the proof of the lemma.

\end{proof}

\section{Special cases}\label{sec-special}
\subsection{Domains birational to the unit polydisc} When $\det B=1$, the map $\phi_A:\Omega\to \Uu_B $ of Proposition~\ref{prop:covering-mono-polyh} is a biholomorphism, and the formula \eqref{eq-kernel2} takes a simple form:
\begin{prop}
    \label{prop-new-biholo}
      Let $B\in \Z^{n\times n}$ be such that $\det B=1$. Then 
   \begin{equation}\label{eq-izzy}
     K_{\Uu_B}(z,w) = \frac{1}{\pi^n} \frac{t^{\one \abs{B}-\one}}{\prod\limits_{j=1}^{n} (t^{(b_-)^j}-t^{(b_+)^j})^2}, \quad \text{ with } \abs{B}=B_{+}+B_{-}.
    \end{equation}
\end{prop}
\begin{proof}
When $\det B=1$, the denominator of \eqref{eq-kernel2} and \eqref{eq-izzy} are the same, so we need to show that $ \sum_{\nu\in \N^{1\times n}} C_B(\nu) t^\nu=  t^{\one |B| - \one}.$
We use the bounds \eqref{eq-nujbounds} to find the values of $\nu$ for which $C_B(\nu)\not=0$. Using \eqref{eq-xij}, we have, for each $1\leq j \leq n$, that
$\displaystyle{ \xi_j=\left\lceil\sum\limits_{k=1}^n\abs{b_j^k} \right\rceil= \sum\limits_{k=1}^n\abs{b_j^k},}$
so the inequalities \eqref{eq-nujbounds} become
\[ -1+ \xi_j\leq \nu_j \leq 2 \sum_{k=1}^n \abs{b_j^k}-1 -\sum\limits_{k=1}^n\abs{b_j^k}= -1+ \xi_j, \]
which means that if $C_B(\nu)\not=0$, we have  $\nu_j=-1+ \sum_{k=1}^n\abs{b_j^k}$,
i.e., $\nu= \one \abs{B}-\one.$  Notice that if $\det B=1$, then $\adj B= \det B \cdot B^{-1}= B^{-1}$, so $B[\adj B]_j=e_j$, where $e_j$ is the column vector with zeros everywhere except in the $j$-th spot. Therefore, we have by \eqref{eq-CB}
\begin{align*}
    C_B(\one \abs{B}-\one)&= \prod\limits_{j=1}^n \dd_1\left(\left(\one (B_++B_-)-\one-2\one B_{-}+\one\right)[\adj B]_j-1\right)\\
    &=\prod\limits_{j=1}^n \dd_1\left(\one B[\adj B]_j-1\right)=\prod\limits_{j=1}^n \dd_1\left(\one e_j-1\right)= \prod\limits_{j=1}^n \dd_1(0)=\prod\limits_{j=1}^n 1 =1,
\end{align*}
using \eqref{eq-Dk} to compute $\dd_1(0)$. Therefore
$\displaystyle{  \sum_{\nu\in \N^{1\times n}} C_B(\nu) t^\nu= C_B(\one |B| - \one) t^{\one |B| - \one} = t^{\one |B| - \one}.}$
\end{proof}

\subsection{Almughrabi's formula for \texorpdfstring{$n=2$}{n2}} We will now recapture the following result:
\begin{prop}[see \cite{almughrabi}]\label{prop-rasha} Let $B\in \Z^{2\times 2}$ satisfy  \eqref{eq-detb1} and \eqref{eq-gcdbj1}. Then denoting $t_j=p_j\cdot\ol{q_j}, j=1,2$, we have
    \begin{equation}
        \label{eq-rashakernel}
        K_{\Uu_B}(p,q)=\frac{1}{\pi^2\cdot \det A}\cdot \frac{g(t_1,t_2)}{\left(t_2^{a_2^1}-t_1^{a_2^2}\right)^2\left(t_1^{a_1^2}-t_2^{a_1^1}\right)^2},
    \end{equation}
    where $a^j_k$ indicates the element in the adjugate matrix $A=\adj B$ of $B$  at the $j$-th row and $k$-th column. The numerator $g(t_1,t_2)$ is a polynomial given by
        \begin{equation}
            \label{eq-rashanum}
            g(t_1,t_2)=\sum\limits_{\nu\in \N^{1\times 2}}\dd_{\det A}(\zeta_1(\nu))\dd_{\det A}(\zeta_2(\nu))t^{\nu_1}t^{\nu_2}
        \end{equation}
    where 
    $ \zeta_j(\nu)= a^1_j\nu_1+ a^2_j\nu_2 - 2(a^2_1a^1_j+ a^1_2a^2_j)+(a^1_j+a^2_j-1)$.
   \end{prop}

\begin{proof}
Notice that
$    A=\begin{pmatrix}
    a^1_1 & a^1_2\\a^2_1 & a^2_2
\end{pmatrix} = \adj B = \begin{pmatrix}
    b^2_2 & -b^1_2 \\ - b^2_1 & b^1_1
\end{pmatrix}.
$
By part (1) of Proposition~\ref{prop:covering-mono-polyh}, we have $A=\adj B\succeq 0$, i.e. the entries of $A$ are nonnegative. Consequently, we have from the above that $ b^1_1\geq 0, b^2_2\geq 0,   b^1_2\leq 0$  and  $b^2_1 \leq 0.$ Consequently, we have 
\[ B_+= \begin{pmatrix}
    b^1_1 &0\\ 0 & b^2_2
\end{pmatrix} , \quad B_-= \begin{pmatrix}
     0 & -b^1_2 \\ - b^2_1 & 0
\end{pmatrix}= \begin{pmatrix}
    0 & a^1_2\\a^2_1 &0
\end{pmatrix}.\]

Therefore, we have, for $\nu \in \Z^{1\times 2}$,
\begin{align}
    \nu-2\one B_-+\one&=(\nu_1, \nu_2) - 2 (1,1) \begin{pmatrix}
     0 & a^1_2 \\ a^2_1 & 0
\end{pmatrix} +(1,1)\nonumber\\
&= (\nu_1- 2a^2_1+1, \nu_2-2a^1_2+1)\label{eq-int1}.
\end{align}
Using the above computations, with $n=2$, the formula \eqref{eq-kernel2} becomes
\begin{align}
      K_{\Uu_B}(p,q)&=\frac{1}{\pi^2\cdot {\det B}}\cdot \frac{\sum\limits_{\nu\in \N^{1\times 2}} C_B(\nu)t^\nu}{ (t^{ (b_-)^1}-t^{(b_+)^1})^2 (t^{ (b_-)^2}-t^{(b_+)^2})^2}\nonumber\\
      &= \frac{1}{\pi^2\cdot {\det B}}\cdot \frac{\sum\limits_{\nu\in \N^{1\times 2}} \prod\limits_{j=1}^2\dd_{\det B}(\zeta_j(\nu))t^\nu}{ (t^{ (0, -b^1_2)}-t^{(b^1_1,0)})^2 (t^{ (-b^2_1,0)}-t^{(0, b^2_2)})^2}\nonumber\\
      &= \frac{1}{\pi^2\cdot {\det B}}\cdot \frac{\sum\limits_{\nu\in \N^{1\times 2}} \dd_{\det B}(\zeta_1(\nu))
      \cdot \dd_{\det B}(\zeta_2(\nu)) t_1^{\nu_1}t_2^{\nu_2}}{ \left(t_2^{-b^1_2}-t_1^{b^1_1}\right)^2 \left(t_1^{-b^2_1}-t_2^{b^2_2}\right)^2}\label{eq-kernelforn=2}\\
       &= \frac{1}{\pi^2\cdot {\det A}}\cdot \frac{\sum\limits_{\nu\in \N^{1\times 2}} \dd_{\det A}(\zeta_1(\nu))
      \cdot \dd_{\det A}(\zeta_2(\nu)) t_1^{\nu_1}t_2^{\nu_2}}{ \left(t_2^{a^1_2}-t_1^{a^2_2}\right)^2 \left(t_1^{a^2_1}-t_2^{a^1_1}\right)^2} \label{eq-kernelforn=2new},
\end{align}
where \eqref{eq-kernelforn=2new} is obtained from \eqref{eq-kernelforn=2} using the fact that $\det A= (\det B)^{n-1}=\det B$ for $n=2$, the relation $A=\adj B$, and where, 
using \eqref{eq-CB} and \eqref{eq-int1}, we have
\begin{align*}
    \zeta_j(\nu)&=\left(\nu-2\one B_{-}+\one\right)[\adj B]_j-1=(\nu_1- 2a^2_1+1, \nu_2-2a^1_2+1)\begin{pmatrix}
        a^1_j\\a^2_j
    \end{pmatrix} -1\\
    &=a^1_j\nu_1+ a^2_j\nu_2 - 2(a^2_1a^1_j+ a^1_2a^2_j)+(a^1_j+a^2_j-1).
\end{align*}
\end{proof}

\subsection{Comments on general formulas for \texorpdfstring{$n\geq 3$}{ngeq3}}
Given the matrix $B\in \Z^{2\times 2}$ defining a two-dimensional monomial polyhedron $\Uu_B$,  using part (1) of Proposition~\ref{prop:covering-mono-polyh}, we have
$B_+=
\begin{pmatrix}
    b_1^1 & 0 \\ 0 & b_2^2
\end{pmatrix}$ and 
$B_-=\begin{pmatrix}
    \phantom- 0 & -b^1_2 \\ -b^2_1 & \phantom- 0
\end{pmatrix}$.
So the decomposition into the positive and negative parts of the matrix $B$ is unique, and this leads to the unique expression \eqref{eq-rashakernel} for the Bergman kernel in its canonical form as the ratio of two coprime polynomials. 

If $n\geq 3$, this is not the case. For example, consider the matrices
$$\begin{pmatrix}
    1&0&-1\\0&1&\phantom- 0\\0&0&\phantom- 1
\end{pmatrix} \text{  and  } \begin{pmatrix}
    1&-1& \phantom- 1\\0&\phantom- 1& -1\\0&\phantom- 0& \phantom- 1
\end{pmatrix}.$$
Both of these have inverses with nonnegative entries and therefore define monomial polyhedra. Thus, the Bergman kernels are represented in canonical form by different formulas. However, it is possible to write
the kernels as \emph{rational} functions using the same general formula in terms of the matrix $B$ as we saw in \eqref{eq-kernel1}. The possible sign patterns can be determined using methods described in \cite{johnson}.
\subsection{Symmetries of \texorpdfstring{$\dd$}{dd}} The following lemma will be used below:
\begin{lem} \label{lem-symmetry} The function $\dd_k$ of \eqref{eq-Dk-gen} has the following properties:
\begin{enumerate}
    \item \label{item-sym1} for integers $k\geq 1$ and $r$, we have $\dd_k(r)= \dd_k(2k-2-r).$
    \item for $k_1,k_2$  positive integers and  each integer $r$ \label{item-Dk1k2} we have 
        $$\dd_{k_1k_2}(k_2(r + 1) - 1)= k_2\dd_{k_1}(r).$$
\end{enumerate}
\end{lem}

\begin{proof} 
Part (1): Using \eqref{eq-Dk-gen}, we have
\begin{align*}
    \sum_{r\in \mathbb{Z}}\dd_k(r)x^r &= \left(\frac{1 - x^k}{1 - x}\right)^2 =\left(\frac{x^k(1-(x^{-1})^{k})}{x(1-x^{-1})}\right)^2\\
    &=x^{2k-2}\cdot \left(\frac{(1-(x^{-1})^{k})}{(1-x^{-1})}\right)^2= x^{2k-2} \sum_{r\in \mathbb{Z}}\dd_k(r)(x^{-1})^r\\
    &=\sum_{r\in \mathbb{Z}}\dd_k(r)x^{2k - 2 - r}= \sum_{\lambda\in\mathbb{Z}}\dd_k(2k - 2 - \lambda)x^{\lambda},
\end{align*}
where the last expression is obtained by reindexing the  sum using $\lambda = 2k - 2 - r$. Upon comparing coefficients of the same degree, we get the result.

Part (2): Again, using \eqref{eq-Dk-gen}, we have
\begin{align*}
    \sum_{\mu \in \mathbb{Z}}\dd_{k_1k_2}(\mu)x^{\mu} &= \left(\frac{1 - x^{k_1k_2}}{1 - x}\right)^2 
    = \left(\frac{1 - (x^{k_2})^{k_1}}{1 - x^{k_2}}\right)^2 \left(\frac{1 - x^{k_2}}{1 - x}\right)^2 \\
    &= \sum_{\lambda_1 \in \mathbb{Z}}\dd_{k_1}(\lambda_1)x^{k_2{\lambda_1}} \cdot\sum_{\lambda_2\in\mathbb{Z}}\dd_{k_2}(\lambda_2)x^{\lambda_2}
\end{align*}
Fix $r\in \Z$ and set
 $\mu = k_2(r + 1) - 1$. Equating coefficients of the same degree, we see that
 \begin{equation}\label{eq-ddconv}
     \dd_{k_1k_2}(k_2(r + 1) - 1) = \sum_{k_2\lambda_1 + \lambda_2 = k_2(r + 1) - 1}\dd_{k_1}(\lambda_1)\dd_{k_2}(\lambda_2).
 \end{equation}
The linear Diophantine equation $k_2\lambda_1 + \lambda_2 = k_2(r + 1) - 1$ for the integer unknowns $\lambda_1,\lambda_2$,
has the general solution $ \lambda_1 =(r+1)-\xi, \quad \lambda_2= -1+k_2\xi, \quad \xi \in \mathbb{Z}.$
 Thus, \eqref{eq-ddconv} becomes
$$\dd_{k_1k_2}(k_2(r + 1) - 1) = \sum_{\xi \in \mathbb{Z}}\dd_{k_1}((r + 1) - \xi)\dd_{k_2}(-1 + k_2\xi).$$
By \eqref{eq-Dk}, the function $\dd_{k_2}$ vanishes outside $[0, 2k_2-2]$.
Therefore, the only value of $\xi$ for which  $\dd_{k_2}(-1 + k_2\xi)$ is  nonzero is $\xi = 1$, so that we have 
$$\dd_{k_1k_2}(k_2(r + 1) - 1) = \dd_{k_1}((r+1)-1)\dd_{k_2}(-1 + k_2\cdot 1) = \dd_{k_1}(r)k_2,$$
where we have used the fact that $\dd_k(k - 1) = k$ from \eqref{eq-Dk}.
\end{proof}

\subsection{Formula for ``signature 1 domains"} Let 
 $k_1,\dots, k_n$ be positive integers with
 $ \gcd(k_1,\dots, k_n)=1.$
In \cite{pjm}, the Bergman kernel of the domain 
\begin{align*}\mathcal{H}_k=
   \{ (z_1,...,z_n)\in \D^n: \abs{z_1}^{k_1}<\abs{z_2}^{k_2}\cdots\abs{z_n}^{k_n}\}
\end{align*}
was computed. We now recapture this result starting from \eqref{eq-kernel2}.
\begin{prop}[See \cite{pjm}] We have
    \label{prop-austin}
    \begin{equation}
    \label{eq-austinkernel}
        K_{\mathcal{H}_k}(p,q)=\frac{1}{\pi^n\cdot L}\cdot \frac{\sum\limits_{\nu\in \N^{1\times n}}E(\nu)t^{\nu}} {\left(\prod\limits_{j=2}^n t_j^{k_j} - t_1^{k_1}\right)^2\cdot \prod\limits_{j=2}^n (1-t_j)^2},
\end{equation}
where $t=(t_1,\dots, t_n)^T$ with $t_j=p_j\cdot\ol{q_j}$, and
\begin{equation}
    \label{eq-CBAust}
    E(\nu)=\dd_{K}(2K- \ell_1(v_1+1)- 1) \cdot \prod\limits_{j=2}^n\dd_{\ell_j}(\ell_j(v_j+1)+\ell_1(v_1+1)-2K-1)
\end{equation}
with
\begin{equation}
    K=\lcm(k_1, ..., k_n), \qquad \ell_a=\frac{K}{k_a} \quad \text{ for  }  1\leq a \leq n, \qquad \text{ and } \quad L=\prod\limits_{a=1}^n\ell_a.
\end{equation}
\end{prop}

\begin{proof}[Proof of Proposition~\ref{prop-austin}]
   
 The domain $\mathcal{H}_k$ is the monomial polyhedron  with the defining matrix
\[B =
\begin{pmatrix}
 k_1 &\rvline & \begin{matrix}   -k_2 & \cdots & -k_n \end{matrix}\\
\hline
\begin{matrix}
  0 \\ \vdots \\ 0
\end{matrix} & \rvline &
  \begin{matrix} &&\\ &\Identity_{n-1} &\\
  &&
  \end{matrix}
 \end{pmatrix}, \text{ so that }  \  A=\adj B = 
\begin{pmatrix} 
 1 & \rvline & \begin{matrix}   k_2 & \cdots & k_n \end{matrix}\\
\hline
\begin{matrix}
  0 \\ \vdots \\ 0
\end{matrix} & \rvline &
  \begin{matrix} &&\\ & k_1\cdot \Identity_{n-1} &\\
  &&
  \end{matrix}
 
\end{pmatrix},\] 
where $\Identity_{n-1}$ is the identity matrix of size $n-1$. Splitting $B=B^+-B^-$, we see that:
    \begin{align*}
    \prod\limits_{j=1}^n \left(t^{ (b_-)^j}-t^{(b_+)^j}\right)^2 
    &=(t^{(0, k_2,...,k_n)} - t^{(k_1, 0, ..., 0)})^2 \cdot (t^{(0,...,0)}-t^{(0,1,0,...,0)})^2 \cdot ... \cdot (t^{(0,...,0)}-t^{(0,..,1)})^2
    \\&=(t_2^{k_2} t_3^{k_3}...t_n^{k_n} - t^{k_1})^2 \cdot (1-t_2)^2 \cdot ... \cdot (1-t_n)^2\\&
    =\left(\prod\limits_{j=2}^n t_j^{k_j} - t_1^{k_1}\right)^2\cdot \prod\limits_{j=2}^n (1-t_j)^2.
\end{align*}
Consequently, the denominator of $K_{\mathcal{H}_k}$ given by \eqref{eq-kernel2} 
is the same as that in \eqref{eq-austinkernel}.
Therefore, the expression given in \eqref{eq-kernel2} will be the same as that in \eqref{eq-austinkernel} provided:
\begin{equation*}
        \frac{C_B(\nu)}{ (\det B)^{n-1}} = \frac{E(\nu)}{ L}, \quad \nu \in \N^{1\times n}.
\end{equation*}
Now, using the fact that $\det B=k_1$, we have
\begin{align*}
    \frac{L}{(\det B)^{n-1}}=\frac{\prod\limits_{a=1}^n(\ell_a)}{(k_1)^{n-1}}
    =\frac{K^n}{\prod\limits_{a=1}^nk_a} \cdot \frac{1}{(k_1)^{n-1}}=\frac{K^n}{(k_1)^{n}} \cdot \frac{1}{\prod\limits_{a=2}^nk_a}=(\ell_1)^n\cdot \frac{1}{\prod\limits_{a=2}^nk_a}.
\end{align*}
So it will suffice to show that
\begin{equation}
    \label{eq-conj}
      (\ell_1)^n\cdot  {C_B(\nu)} = \left({\prod\limits_{a=2}^nk_a}\right)\cdot {E(\nu)}, \quad \nu \in \N^{1\times n}.
\end{equation}

To compute the LHS of \eqref{eq-conj}, denoting by $0_{n-1\times n}$ the matrix of size $n-1\times n$ with zero entries,  we have
\[    \one B_-= (1,\dots, 1) \begin{pmatrix}
        0 & k_2& \cdots & k_n\\ & \phantom{-}\large{0_{n-1\times n}}& 
    \end{pmatrix}= (0,k_2,\dots, k_n).\]
    Therefore, for $\nu\in \N^{1\times n}$,
    $ \nu-2\one B_-+\one= (\nu_1+1, \nu_2-2k_2+1,\dots, \nu_n-2k_n+1),$
    and consequently, using $\det B=1$, we have from \eqref{eq-CB}
    \begin{align}
    \ell_1^n\cdot C_B(\nu)&=  \ell_1^n\prod\limits_{j=1}^n \dd_{k_1}\left(\left(\nu-2\one B_{-}+\one\right)a_j-1\right)
    \nonumber\\&=  \ell_1\cdot\dd_{k_1}((\nu-2\one B_{-}+\one)a_1-1) \cdot  \ell_1^{n-1}\cdot\prod\limits_{j=2}^n \dd_{k_1}\left(\left(\nu-2\one B_{-}+\one\right)a_j-1\right)
\nonumber\\&= \left( \ell_1\dd_{k_1}(\nu_1)\right) \cdot \prod\limits_{j=2}^n \left(\ell_1\cdot\dd_{k_1} ((\nu_1+1)k_j+(\nu_j-2k_j+1)k_1 - 1)\right).\label{eq-carry}
\end{align}
Now, notice that
\begin{subequations}
    \begin{align}
    \dd_{K}(2K- \ell_1(\nu_1+1)- 1)&=\dd_{k_1\ell_1}(2K- 2 - (\ell_1(\nu_1+1)- 1))\label{eq-number}
    \\&=\dd_{k_1\ell_1}(\ell_1(\nu_1+1)- 1) \label{eq-number2}
    \\&=\ell_1\dd_{k_1}(\nu_1)\label{eq-no3},
\end{align}\end{subequations}

where \eqref{eq-number}$\Rightarrow$\eqref{eq-number2} is by part \eqref{item-sym1} of \ref{lem-symmetry}, and  \eqref{eq-number2} $\Rightarrow$\eqref{eq-no3} is by part \eqref{item-Dk1k2} of \ref{lem-symmetry}. Notice also that
\begin{subequations}
    \begin{align}
    &k_j\dd_{\ell_j}(\ell_j(\nu_j+1)+\ell_1(\nu_1+1)-2K-1)\nonumber
    \\&=\dd_{k_j\ell_j}(k_j((\ell_j(\nu_j+1)+\ell_1(\nu_1+1)-2K-1)+1)-1) \label{eq-number3} 
    \\&=\dd_{K}(k_j(\ell_j(\nu_j+1)+\ell_1(\nu_1+1)-2K)-1) \label{eq-number4}
    \\&=\dd_{K}(K(\nu_j+1)+k_j\ell_1(\nu_1+1)-2Kk_j-1) \nonumber 
    \\&=\dd_{k_1\ell_1}(k_1\ell_1(\nu_j+1)+k_j\ell_1(\nu_1+1)-2k_1\ell_1k_j-1) \label{eq-number5} 
    \\&=\dd_{k_1\ell_1}(\ell_1(k_1(\nu_j+1)+k_j(\nu_1+1-2k_1))-1) \nonumber 
    \\&=\ell_1\dd_{k_1}(k_1(\nu_j+1)+k_j(\nu_1+1-2k_1) - 1), \label{eq-number6}
\end{align}\end{subequations}
where in \eqref{eq-number3} and \eqref{eq-number6}, we apply part \eqref{item-Dk1k2} of \ref{lem-symmetry}, and in \eqref{eq-number4}, we use the definition of $\ell_j$ to rewrite $k_j\ell_j$ as $K$, and in \eqref{eq-number5} to rewrite $K$ as $k_1\ell_1$.

Therefore, we have
\begin{align*}
    \eqref{eq-carry}&=  \dd_{K}(2K- \ell_1(\nu_1+1)- 1)\prod\limits_{j=2}^nk_j\dd_{\ell_j}(\ell_j(\nu_j+1)+\ell_1(\nu_1+1)-2K-1)\\
    &=\left(\prod\limits_{a=2}^nk_a\right)\cdot\dd_{K}(2K- \ell_1(\nu_1+1)- 1)\prod\limits_{j=2}^n\dd_{\ell_j}(\ell_j(\nu_j+1)+\ell_1(\nu_1+1)-2K-1)\\
    &=\left({\prod\limits_{a=2}^nk_a}\right)\cdot {E(\nu)},
\end{align*}
completing the proof.
\end{proof}


\subsection{The Park-Zhang formula for Generalized Hartogs Triangles}
Let $p_1,\dots, p_n$ be positive integers such that $\gcd(p_1,\dots,p_n)=1$. In 
\cite{park,zhang1}, the  domain
\begin{equation}
    \label{eq-gp}\mathcal{G}=\mathcal{G}_{p_1,\dots, p_n}=\{ (z_1, \dots, z_n) \in \mathbb{C}^n : \abs{z_1}^{p_1} < \dots < \abs{z_n}^{p_n} < 1\}
\end{equation}
was called the \emph{Generalized Hartogs Triangle}, its Bergman kernel was determined, and the regularity of the Bergman projection in $L^p$-spaces was studied. 
To state their formula for the kernel, we introduce the following notation:
\begin{equation}
    \label{eq-notation1}
    P= \prod_{j=1}^n p_j, \quad p_j' = \frac{P}{p_j}, \quad d_j = \gcd(p_j,p_{j+1}), \quad \text{ with } d_n = p_n.
\end{equation}
In \cite{park,zhang1}, $p_j'$ was denoted as $k_j$. 
Notice that we have
$ p_1p_1' = \dots = p_np_n'.$
Let 
\begin{equation}
    \label{eq-notation2}K = \prod_{j=1}^n p_j' = \frac{P^n}{\prod_{j=1}^n p_j}=P^{n-1},
\end{equation}
and also let for $1\leq j \leq n-1$
$$\displaystyle{k_j^{(j)} = \frac{p_j'}{\gcd(p_j',p_{j+1}')}, \quad  k_{j+1}^{(j)} = \frac{p_{j+1}'}{\gcd(p_j',p_{j+1}')},}$$ where we take $p_{n+1}' = 1$. Since 
$\displaystyle{\gcd(p_j',p_{j+1}')=\gcd\left(\frac{P}{p_j},\frac{P}{p_{j+1}}\right)= \frac{P}{\lcm(p_j,p_{j+1})}},$
we obtain
\begin{equation}
    \label{eq-kjj}
    k_j^{(j)}=  \frac{p_j'}{\gcd(p_j',p_{j+1}')}= \frac{\dfrac{P}{p_j}}{\dfrac{P}{\lcm(p_j,p_{j+1})}}=\frac{\lcm(p_j,p_{j+1})}{p_j}= \frac{p_{j+1}}{\gcd(p_j,p_{j+1})} = \frac{p_{j+1}}{d_j},
\end{equation}
where $p_j$ and $d_j$ are as  in \eqref{eq-notation1}. In a similar manner, one sees
\begin{equation}
    \label{eq-kjjbis}
    k_{j+1}^{(j)} = \frac{p_j}{d_j}.
\end{equation}
\begin{prop}\label{prop-parkzhang} With $t$ as in \eqref{eq-kernel2}:
\begin{equation}
\label{eq-KG}
K_{\mathcal{G}}(p,q) = \frac{\sum\limits_{\alpha_1 = 0}^{N_1}\cdots\sum\limits_{\alpha_n = 0}^{N_n}\nu(P_1)\cdots\nu(P_n)t^{\alpha}}{\pi^n K (1 - t_n)^2\prod\limits_{j=1}^{n-1}\left(t_j^{k_{j + 1}^{(j)}} - t_{j+1}^{k_j^{(j)}}\right)^2},
\end{equation}
where we have, with $1\leq \ j \leq n$,
\begin{enumerate}
    \item $m_{j,j+1} = \lcm(p_j',p_{j+1}')$ (with $m_{n,n+1} = p_n'$),
    \item $\displaystyle{
N_1 = \left\lfloor \frac{2m_{1,2} - 1 - p_1'}{p_1'} \right\rfloor, \quad
N_j = \left\lfloor \frac{2m_{j-1,j} + 2m_{j, j+1} - p_j' - 2}{p_j'} \right\rfloor}$,

\item $\displaystyle{
P_1 = 2m_{1,2} - p_1' + 1 - p_1'\alpha_1, \quad
P_j = 2m_{j, j+1} - p_j' - p_j'\alpha_j + P_{j-1},
}$\label{item-rec}
\item 
$\displaystyle{
    \nu(P_j) = \begin{cases} 
            P_j - 1, & 2 \leq P_j \leq m_{j, j+1} + 1,\\
            2m_{j,j+1} - P_j + 1, & m_{j,j+1} + 2 \leq P_j \leq 2m_{j,j+1},\\
            0, & P_j < 2 \text{ or } P_j > 2m_{j,j+1}.
            \end{cases}
}$
\end{enumerate}
\end{prop}

\begin{proof} From $(2)$ and $(4)$, we see that $\nu(P_j)=0$ unless $0\leq \alpha_j \leq N_j$. Therefore, we can replace the sum in the numerator of \ref{eq-KG} with one over all natural numbers. We let
$\Lambda = \prod\limits_{j = 1}^nd_j, \ d_j' = \frac{\Lambda}{d_j}.$ One sees easily that $\mathcal{G}$ is a monomial polyhedron defined by the upper triangular matrix
$$
B = \begin{pmatrix}
    \frac{p_1}{d_1} & -\frac{p_2}{d_1} & & & \\
     & \frac{p_2}{d_2} & -\frac{p_3}{d_2} & & \\
     & & \ddots & \ddots & \\
     & & & \frac{p_{n-1}}{d_{n-1}} & -\frac{p_n}{d_{n-1}} \\
     & & & & \frac{p_n}{d_n}
\end{pmatrix}, \quad \text{ so that }\quad
\adj B = 
\begin{pmatrix}
    \frac{p_1'}{d_1'} & \frac{p_1'}{d_2'} & \cdots & \cdots & \frac{p_1'}{d_n'}\\
     & \frac{p_2'}{d_2'} & \cdots & \cdots & \frac{p_2'}{d_n'}\\
     & & \ddots & \ddots & \vdots \\
     & & & \ddots & \vdots \\
     & & & & \frac{p_n'}{d_n'}
\end{pmatrix}.
$$
Therefore we have   $B_+=\diag\left( \frac{p_1}{d_1},\dots, \frac{p_n}{d_n}\right)$, and
$$
B_- = 
\begin{pmatrix}
    0 & \frac{p_2}{d_1} & & & \\
     & \ddots & \ddots & & \\
     & & \ddots & \ddots & \\
     & & & \ddots & \frac{p_n}{d_{n-1}} \\
     & & & & 0 
\end{pmatrix}, \quad \text{ so } \one B_-= \left(0, \frac{p_2}{d_1},\frac{p_3}{d_2}, \dots,\frac{p_n}{d_{n-1}}\right).
$$
From the expressions for 
$B_+$ and $B_-$, we have
\begin{align*}
    \prod\limits_{j=1}^n \left(t^{ (b_-)^j}-t^{(b_+)^j}\right)^2&=\prod_{j=1}^{n-1}\left(t_{j+1}^{\frac{p_{j+1}}{d_j}} - t_j^{\frac{p_j}{d_j}}\right)^2\cdot \left(1 - t_n^{\frac{p_n}{d_n}}\right)^2= \left(1 - t_n\right)^2\prod_{j=1}^{n-1}\left(t_{j+1}^{\frac{p_{j+1}}{d_j}} - t_j^{\frac{p_j}{d_j}}\right)^2\\
    &=(1 - t_n)^2\prod_{j=1}^{n-1}\left(t_j^{k_{j + 1}^{(j)}} - t_{j+1}^{k_j^{(j)}}\right)^2 \text{using \eqref{eq-kjj} and \eqref{eq-kjjbis},}
\end{align*}
where in the first line we have used that $d_n=p_n.$
Since $\det(B) = \frac{P}{\Lambda}$,  \eqref{eq-kernel2} gives
$$K_{\mathcal{G}}(p,q) =\frac{1}{\pi^n\cdot {\left(\dfrac{P}{\Lambda}\right)^{n-1}}}\cdot \frac{\sum\limits_{\alpha\in \N^{1\times n}}C_B(\alpha)t^{\alpha}}{\prod\limits_{j=1}^n (t^{ (b_-)^j}-t^{(b_+)^j})^2},$$
which, since $K=P^{n-1}$, coincides with the Park-Zhang expression for the kernel
\eqref{eq-KG} if and only if  for each $\alpha\in \N^{1\times n}$ 
\begin{equation}
    \label{eq-conj1}
    \Lambda^{n-1} \cdot C_B(\alpha) = \nu(P_1)\cdots\nu(P_n).
\end{equation}
We now claim that:
\begin{equation}\label{eq-nu_to_D}
\nu(P_j)  =
     d_j'\cdot \dd_{\frac{P}{\Lambda}}\left(\frac{1}{d_j'}\left[\left(\sum_{i=1}^jp_i'\alpha_i\right) - 2 \left(\sum_{i=1}^{j-1}\frac{P}{d_i}\right) + \left(\sum_{i=1}^jp_i'\right) \right] - 1\right).
\end{equation}
To see this, first observe that for $1\leq j\leq n$, we have
    $m_{j,j+1} = \lcm\left(p_j',p_{j+1}'\right) = \lcm\left(\frac{P}{p_j},\frac{P}{p_{j+1}}\right) = \frac{P}{d_j},$
    and also  by \eqref{eq-Dk}:
    \begin{align}
    \dd_{m_{j,j + 1}}(P_j - 2) &= \begin{cases} 
            P_j - 1, & 0\leq P_j - 2 \leq m_{j,j+1}-1,\\
            2m_{j,j+1} - P_j + 1, & m_{j,j+1} \leq P_j - 2 \leq 2m_{j,j+1} - 2,\\
            0, & P_j - 2 < 0 \text{ or } P_j - 2 > 2m_{j,j+1}-2
            \end{cases} \nonumber \\
            &=\nu(P_j) \label{eq-D_as_nu}.
    \end{align}
Using our description of $m_{l,l+1}$, we have 
$P_1 = -p_1'\alpha_1 + 2\frac{P}{d_1} - p_1' + 1$,
so  upon expanding the recursion in part (\ref{item-rec}), we find
$P_j = -\left(\sum_{i=1}^jp_i'\alpha_i\right) + 2 \left(\sum_{i=1}^j\frac{P}{d_i}\right) - \left(\sum_{i=1}^jp_i'\right) + 1$. Therefore, we conclude, applying \ref{eq-D_as_nu}, that
    \begin{align}
    \nu(P_j) 
            &=\dd_{\frac{P}{d_j}} \left(-\left(\sum_{i=1}^jp_i'\alpha_i\right) + 2 \left(\sum_{i=1}^j\frac{P}{d_i}\right) - \left(\sum_{i=1}^jp_i'\right) - 1\right)\nonumber \\
            &= \dd_{\frac{P}{d_j}} \left(2\frac{P}{d_j} - 2 -\left[\left(\sum_{i=1}^jp_i'\alpha_i\right) - 2 \left(\sum_{i=1}^{j-1}\frac{P}{d_i}\right) + \left(\sum_{i=1}^jp_i'\right) - 1\right] \right)\nonumber \\
            &= \dd_{\frac{P}{d_j}} \left(\left(\sum_{i=1}^jp_i'\alpha_i\right) - 2 \left(\sum_{i=1}^{j-1}\frac{P}{d_i}\right) + \left(\sum_{i=1}^jp_i'\right) - 1\right) \label{ex-sym1} \\
            &= \frac{d_j'}{d_j'} \dd_{\frac{P}{d_j}}\left(\left(\sum_{i=1}^jp_i'\alpha_i\right) - 2 \left(\sum_{i=1}^{j-1}\frac{P}{d_i}\right) + \left(\sum_{i=1}^jp_i'\right) - 1\right) \nonumber \\
            &= d_j'\cdot \dd_{\frac{P}{d_jd_j'}}\left(\frac{1}{d_j'}\left[\left(\sum_{i=1}^jp_i'\alpha_i\right) - 2 \left(\sum_{i=1}^{j-1}\frac{P}{d_i}\right) + \left(\sum_{i=1}^jp_i'\right) - 1 + 1\right] - 1\right) \label{ex-sym2} \\
            &= d_j'\cdot \dd_{\frac{P}{\Lambda}}\left(\frac{1}{d_j'}\left[\left(\sum_{i=1}^jp_i'\alpha_i\right) - 2 \left(\sum_{i=1}^{j-1}\frac{P}{d_i}\right) + \left(\sum_{i=1}^jp_i'\right)\right] - 1\right), \nonumber
\end{align}
where \ref{ex-sym1} follows from part \ref{item-sym1} of \ref{lem-symmetry} and \ref{ex-sym2} from part \ref{item-Dk1k2} of \ref{lem-symmetry}.
Using the expression for $\one B_-$ and that $\det B= \frac{P}{\Lambda}$, for $\alpha \in \mathbb{Z}^{1\times n}$, we have by \eqref{eq-CB}:
\begin{align*}
C_B(\alpha) &= \prod\limits_{j=1}^n \dd_{\frac{P}{\Lambda}}\left(\left(\alpha-2\one B_{-}+\one\right)[\adj B]_j-1\right) \\
    &= \prod\limits_{j=1}^n \dd_{\frac{P}{\Lambda}}\left(\left(\alpha_1 + 1, \alpha_2 - 2\frac{p_2}{d_2} + 1, \dots, \alpha_n - 2\frac{p_n}{d_{n-1}} +1 \right)[\adj B]_j-1\right) \\
    &= \prod\limits_{j=1}^n \dd_{\frac{P}{\Lambda}}\left(\frac{1}{d_j'}\left(p_1'(\alpha_1 + 1) + p_2'(\alpha_2 - 2\frac{p_2}{d_2} + 1) + \dots + p_j'(\alpha_j - 2\frac{p_j}{d_{j-1}} +1) \right) - 1 \right) \\
    &= \prod\limits_{j=1}^n \dd_{\frac{P}{\Lambda}}\left(\frac{1}{d_j'}\left((p_1'\alpha_1 + \dots + p_j'\alpha_j) - 2P\left(\frac{1}{d_1} + \dots + \frac{1}{d_{j-1}}\right) + (p_1' + \dots + p_j') \right) - 1 \right)\\
    &=  \prod\limits_{j=1}^n \frac{1}{d_j'} \nu(P_j)= \frac{1}{\Lambda^{n-1}}\cdot \prod\limits_{j=1}^n \nu(P_j), \text{ using \ref{eq-nu_to_D}},
\end{align*}
which is what we wanted to prove.
\end{proof}
\bibliographystyle{alpha}
\bibliography{monomial}

\begin{thebibliography}{CKMM20}

\bibitem[Alm23]{almughrabi}
Rasha Almughrabi.
\newblock Bergman kernels of two dimensional monomial polyhedra, 2023.
\newblock To appear in \emph{Complex Analysis and Operator Theory}; avaiable
  online at \url{https://arxiv.org/abs/2303.14268}.

\bibitem[BCEM22]{summer20}
Chase Bender, Debraj Chakrabarti, Luke Edholm, and Meera Mainkar.
\newblock {$L^p$}-regularity of the {B}ergman projection on quotient domains.
\newblock {\em Canad. J. Math.}, 74(3):732--772, 2022.

\bibitem[Bel82]{belltransactions}
Steven~R. Bell.
\newblock The {B}ergman kernel function and proper holomorphic mappings.
\newblock {\em Trans. Amer. Math. Soc.}, 270(2):685--691, 1982.

\bibitem[Bel84]{bellrat}
Steven Bell.
\newblock Proper holomorphic mappings that must be rational.
\newblock {\em Trans. Amer. Math. Soc.}, 284(1):425--429, 1984.

\bibitem[Bre55]{bremmerman}
H.~J. Bremermann.
\newblock Holomorphic continuation of the kernel function and the {B}ergman
  metric in several complex variables.
\newblock In {\em Lectures on functions of a complex variable}, pages 349--383.
  University of Michigan Press, Ann Arbor, Mich., 1955.

\bibitem[CE23]{mbp}
Debraj {Chakrabarti} and Luke~D. {Edholm}.
\newblock {Projections onto $L^p$-Bergman spaces of Reinhardt Domains}.
\newblock {\em arXiv e-prints}, page arXiv:2303.10005, March 2023.

\bibitem[CEM19]{ChEdMc19}
D.~Chakrabarti, L.~D. Edholm, and J.~D. McNeal.
\newblock Duality and approximation of {B}ergman spaces.
\newblock {\em Adv. Math.}, 341:616--656, 2019.

\bibitem[Che17]{Chen17}
Liwei Chen.
\newblock The {$L^p$} boundedness of the {B}ergman projection for a class of
  bounded {H}artogs domains.
\newblock {\em J. Math. Anal. Appl.}, 448(1):598--610, 2017.

\bibitem[CKMM20]{pjm}
Debraj Chakrabarti, Austin Konkel, Meera Mainkar, and Evan Miller.
\newblock Bergman kernels of elementary {R}einhardt domains.
\newblock {\em Pacific J. Math.}, 306(1):67--93, 2020.

\bibitem[CKY20]{CKY19}
Liwei Chen, Steven~G. Krantz, and Yuan Yuan.
\newblock {$L^p$} regularity of the {B}ergman projection on domains covered by
  the polydisc.
\newblock {\em J. Funct. Anal.}, 279(2):108522, 20, 2020.

\bibitem[CZ16]{chakzeytuncu}
Debraj Chakrabarti and Yunus~E. Zeytuncu.
\newblock {$L^p$} mapping properties of the {B}ergman projection on the
  {H}artogs triangle.
\newblock {\em Proc. Amer. Math. Soc.}, 144(4):1643--1653, 2016.

\bibitem[DM23]{dallara}
G.~Dall'Ara and A.~Monguzzi.
\newblock Nonabelian ramified coverings and {$L^p$}-boundedness of {B}ergman
  projections in {$\mathbb{C}^2$}.
\newblock {\em J. Geom. Anal.}, 33(2):Paper No. 52, 28, 2023.

\bibitem[Edh16a]{Edh16}
Luke~D. Edholm.
\newblock {B}ergman theory of certain generalized {H}artogs triangles.
\newblock {\em Pacific J. Math.}, 284(2):327--342, 2016.

\bibitem[Edh16b]{lukethesis}
Luke~David Edholm.
\newblock {\em The {B}ergman kernel of fat {H}artogs triangles}.
\newblock ProQuest LLC, Ann Arbor, MI, 2016.
\newblock Thesis (Ph.D.)--The Ohio State University.

\bibitem[EM16]{EdhMcN16}
L.~D. Edholm and J.~D. McNeal.
\newblock The {B}ergman projection on fat {H}artogs triangles: ${L}^p$
  boundedness.
\newblock {\em Proc. Amer. Math. Soc.}, 144(5):2185--2196, 2016.

\bibitem[EM17]{EdMc1}
L.~D. Edholm and J.~D. McNeal.
\newblock Bergman subspaces and subkernels: degenerate {$L^p$} mapping and
  zeroes.
\newblock {\em J. Geom. Anal.}, 27(4):2658--2683, 2017.

\bibitem[EM20]{EdhMcN20}
L.~D. Edholm and J.~D. McNeal.
\newblock Sobolev mapping of some holomorphic projections.
\newblock {\em J. Geom. Anal.}, 30(2):1293--1311, 2020.

\bibitem[EXX21]{eben}
Peter {Ebenfelt}, Ming {Xiao}, and Hang {Xu}.
\newblock {Algebraic Bergman kernels and finite type domains in
  $\mathbb{C}^2$}.
\newblock {\em arXiv e-prints}, page arXiv:2111.07175, November 2021.

\bibitem[Fu94]{fu94}
Siqi Fu.
\newblock A sharp estimate on the {B}ergman kernel of a pseudoconvex domain.
\newblock {\em Proc. Amer. Math. Soc.}, 121(3):979--980, 1994.

\bibitem[Fu14]{fu}
Siqi Fu.
\newblock Estimates of invariant metrics on pseudoconvex domains near
  boundaries with constant {L}evi ranks.
\newblock {\em J. Geom. Anal.}, 24(1):32--46, 2014.

\bibitem[HKZ00]{zhubergman}
Haakan Hedenmalm, Boris Korenblum, and Kehe Zhu.
\newblock {\em Theory of {B}ergman spaces}, volume 199 of {\em Graduate Texts
  in Mathematics}.
\newblock Springer-Verlag, New York, 2000.

\bibitem[Joh83]{johnson}
Charles~R. Johnson.
\newblock Sign patterns of inverse nonnegative matrices.
\newblock {\em Linear Algebra Appl.}, 55:69--80, 1983.

\bibitem[Kra13]{krantzbergman}
Steven~G. Krantz.
\newblock {\em Geometric analysis of the {B}ergman kernel and metric}, volume
  268 of {\em Graduate Texts in Mathematics}.
\newblock Springer, New York, 2013.

\bibitem[NP09]{nagelduke}
Alexander Nagel and Malabika Pramanik.
\newblock Maximal averages over linear and monomial polyhedra.
\newblock {\em Duke Math. J.}, 149(2):209--277, 2009.

\bibitem[NP20]{nagelpramanik}
Alexander Nagel and Malabika Pramanik.
\newblock Bergman spaces under maps of monomial type.
\newblock {\em The Journal of Geometric Analysis}, 2020.

\bibitem[Par18]{park}
Jong-Do Park.
\newblock The explicit forms and zeros of the {B}ergman kernel for
  3-dimensional {H}artogs triangles.
\newblock {\em J. Math. Anal. Appl.}, 460(2):954--975, 2018.

\bibitem[Zey20]{zeytuncu20}
Yunus~E. Zeytuncu.
\newblock A survey of the {$L^p$} regularity of the {B}ergman projection.
\newblock {\em Complex Anal. Synerg.}, 6(2):Paper No. 19, 7, 2020.

\bibitem[Zha21a]{zhang1}
Shuo Zhang.
\newblock {$L^p$} boundedness for the {B}ergman projections over
  {$n$}-dimensional generalized {H}artogs triangles.
\newblock {\em Complex Var. Elliptic Equ.}, 66(9):1591--1608, 2021.

\bibitem[Zha21b]{zhang2}
Shuo Zhang.
\newblock Mapping properties of the {B}ergman projections on elementary
  {R}einhardt domains.
\newblock {\em Math. Slovaca}, 71(4):831--844, 2021.

\end{thebibliography}
\end{document}